\documentclass[11pt]{amsart}
\usepackage[usenames]{color}
\usepackage{fullpage}
\usepackage{amscd}
\usepackage{amssymb, latexsym}
\usepackage[all,2cell,ps]{xy}
\usepackage{mathdots}

\theoremstyle{plain}
\newtheorem{thm}{Theorem}[section]
\newtheorem{theorem}[thm]{Theorem}

\newtheorem{corollary}[thm]{Corollary}

\newtheorem{lemma}[thm]{Lemma}

\newtheorem{problem}[thm]{Problem}

\newtheorem{prop}[thm]{Proposition}
\newtheorem{proposition}[thm]{Proposition}

\newtheorem{fact}[thm]{Fact}

\theoremstyle{definition}
\newtheorem{de}[thm]{Definition}

\newtheorem{rem}[thm]{Remark}
\newtheorem{remark}[thm]{Remark}
\newtheorem{example}[thm]{Example}

\newtheorem{algorithm}[thm]{Algorithm}

\newcommand{\Z}{\mathbb{Z}}

\newcommand{\Soc}{\mathop{\mathrm{Soc}}}

\newcommand{\id}{\mathrm{id}}

\newcommand{\aut}[1]{\mathrm{Aut}(#1)}

\newcommand{\Ker}{\mathop{\mathrm{Ker}}}

\newcommand{\Mlt}{\mathop{\mathcal{G}}}
\newcommand{\Aut}{\mathop{\mathrm{Aut}}}

\newcommand{\LMlt}{\mathop{\mathcal{G}_{\ell}}}
\newcommand{\RMlt}{\mathop{\mathcal{G}_{r}}}

\numberwithin{equation}{section}

\begin{document}

\title{Skew left braces and 2-reductive 
solutions
\\of the Yang-Baxter equation}

\author{P\v remysl Jedli\v cka}
\author{Agata Pilitowska}

\address{(P.J.) Department of Mathematics, Faculty of Engineering, Czech University of Life Sciences, Kam\'yck\'a 129, 16521 Praha 6, Czech Republic}
\address{(A.P.) Faculty of Mathematics and Information Science, Warsaw University of Technology, Koszykowa 75, 00-662 Warsaw, Poland}

\email{(P.J.) jedlickap@tf.czu.cz}
\email{(A.P.) agata.pilitowska@pw.edu.pl}



\date{\today}

\begin{abstract}
We study 2-reductive non-involutive non-degenerate set-theoretic solutions of the Yang-Baxter equation. We give a combinatorial
construction of any such solution of any (even infinite) size. We also prove that
solutions associated to a skew left brace are 2-reductive if and only if the skew left brace is nilpotent of class~$2$.
Moreover, all such skew left braces are
actually bi-skew left braces. We focus on these structures and we give several equivalent properties characterizing solutions associated to bi-skew left braces.

\end{abstract}

\maketitle

\section{Introduction}
The Yang-Baxter equation is a fundamental equation occurring in mathematical physics. It appears, for example, in integrable models in statistical mechanics, quantum field theory or Hopf algebras~(see e.g. \cite{Jimbo, K}). Searching for its solutions has been absorbing researchers for many years.

Let us recall that, for a vector space $V$, a {\em solution of the Yang--Baxter equation} is a linear mapping $r:V\otimes V\to V\otimes V$ 
 such that
\begin{align*}
(id\otimes r) (r\otimes id) (id\otimes r)=(r\otimes id) (id\otimes r) (r\otimes id).
\end{align*}

Description of all possible solutions seems to be extremely difficult and therefore
there were some simplifications introduced by Drinfeld in \cite{Dr90}.
Let  $X$ be a basis of the space $V$ and let $\sigma:X^2\to X$ and $\tau: X^2\to X$ be two mappings. We say that $(X,\sigma,\tau)$ is a {\em set-theoretic solution of the Yang--Baxter equation} if
the mapping 
$$x\otimes y \mapsto \sigma(x,y)\otimes \tau(x,y)$$ extends to a solution of the Yang--Baxter
equation. It means that $r\colon X^2\to X^2$, where $r=(\sigma,\tau)$,  is a bijection and
satisfies the \emph{braid relation}:
\begin{equation}\label{eq:braid}
(id\times r)(r\times id)(id\times r)=(r\times id)(id\times r)(r\times id).
\end{equation}

A solution $(X,r)=(X,\sigma,\tau)$ is called {\em non-degenerate} if the mappings $\sigma_x=\sigma(x,\_)$ and $\tau_y=\tau(\_\,,y)$ are bijections,
for all $x,y\in X$.
A solution 
is {\em involutive} if $r^2=\mathrm{id}_{X^2}$, i.e. for each $x,y\in X$, $\tau_y(x)=\sigma_{\sigma_x(y)}^{-1}(x)$ and $\sigma _x(y)=\tau^{-1}_{\tau_y(x)}(y)$. Moreover, it is
\emph{square free} if $r(x,x)=(x,x)$, for every $x\in X$.

All solutions we study in this paper are set-theoretic and non-degenerate and we will call them simply \emph{solutions}.
Although researchers usually focus on finite solutions only, in our paper
the set $X$ can be of arbitrary cardinality. 

In the last decade, the main interest of researchers lied in the study
of involutive solutions, mainly due to the fact that every involutive solution embeds
into a solution associated to a ring-like structure called {\em left brace} \cite{Rump07A, CJO14}.
A special emphasis was taken onto so-called {\em multipermutation solution}
since the property resembles nilpotency and it is connected to the nilpotency of left braces (see e.g. \cite{GIC12, S18}).

In~2017 Guarnieri and Vendramin~\cite{GV} generalized the concept of left braces
introducing the skew left braces (Definition~\ref{def:skew_left_brace}). These structures yield non-involutive solutions and we can also
define {\em multipermutation} solutions (Definition~\ref{ret}).
Due to Bachiller~\cite{B18}, we have,
analogously as in the involutive case, that the multipermutation property
of solutions
corresponds to the nilpotency property
of skew left braces (Corollary~\ref{cor:RetSoc}).
The path is therefore cleared for
researchers to study non-involutive multipermutation solutions as 
Vendramin formulated in \cite[Problem 23]{V19}.
Of course, such a study is more complicated in the non-involutive case since we now need to consider two almost independent mappings
$\sigma$ and $\tau$, whereas
in the involutive case one of the mappings was defined uniquely from the other one.

The least complicated involutive solutions are so called $2$-{\it reductive} ones. The authors together with Zamojska-Dzienio showed
in~\cite{JPZ20a} that the nature
of such solutions
is a combinatorial one rather than algebraic and
a~combinatorial construction of any such solution was presented there.
The notion of $2$-reductivity can be
naturally generalized for non-involutive solutions too, only instead of one identity we need four identities (Definition~\ref{def:2red}). It then turns out that the
combinatorial construction from~\cite{JPZ20a} translates straightforwardly for non-involutive solutions. 
We also 
prove, analogously as in the involutive case, that a skew left brace
yields a $2$-reductive solution if and only if it is nilpotent of class~$2$ (Theorem \ref{thm:2-red_brace}).

The $2$-reductive involutive solutions
appeared in a different context as well. Stefanello and Trappeniers studied in~\cite{ST} so called
{\em bi-skew left braces} (Definition~\ref{def:SLB}) and their connections with solutions. They
proved that an involutive solution
associated to a bi-skew left brace
is always $2$-reductive. In the
non-involutive case, we find in
Theorem~\ref{thm:dislbilred}
that the solution
associated to a bi-skew left brace
is left distributive, a property
introduced (in the context of solutions) in~\cite{JPZ20}. On the other hand, Bardakov, Neshchadim and Yadav investigated in \cite{BNY} skew left braces $(B,\cdot,\circ)$ for which the mapping $\lambda$ is a homomorphism of groups $(B,\cdot)$ and $\aut{B,\cdot}$ ($\lambda$-\emph{homomorphic skew left braces}). In Proposition  \ref{prop:4ekv2red} we showed that each skew left brace with $2$-reductive associated solution is $\lambda$-homomorphic.

The paper is organized as follows: in Section~\ref{sec:prelim} we recall basic definitions and properties of solutions. In~Section~\ref{sec:2red} we introduce the definition of $2$-reductivity and we present the combinatorial construction of $2$-reductive solutions (Theorem \ref{th:2red}).
In~Section~\ref{sec:skew} we recall the
definition of a skew left brace, of a bi-skew left brace and of their associated solutions.
We also recall many properties of these structures.
In Section~\ref{sec:distr} we characterize
the solutions associated to bi-skew left braces using several equivalent properties (Theorem~\ref{thm:dislbilred}) 
and we show the impact of these properties
on bi-skew left braces.
Finally, in Section~\ref{sec:last} we  
describe the skew left braces that yield $2$-reductive solutions; more precisely we give an equivalent characterization for each of the four identities defining the $2$-reductivity separately (Proposition \ref{prop:4ekv2red}) as well as a characterization for all the four properties combined (Theorem \ref{thm:2-red_brace}).

\section{Preliminaries}\label{sec:prelim}

Since $r\colon X^2\to X^2$ is a bijective map there is $r^{-1}\colon X^2\to X^2$. It is also true that $(X,r^{-1})$ is a solution. Let for $x\in X$, $\hat{\sigma}_x,\hat{\tau}_x\colon X\to X$ be such that $r^{-1}(x,y)=(\hat{\sigma}_x(y),\hat{\tau}_y(x))$. Clearly, we have that for $x,y\in X$:
\begin{align*}
&(x,y)=rr^{-1}(x,y)=r(\hat{\sigma}_x(y),\hat{\tau}_y(x))=(\sigma_{\hat{\sigma}_x(y)}\hat{\tau}_y(x),\tau_{\hat{\tau}_y(x)}\hat{\sigma}_x(y)), \quad {\rm and}\\
&(x,y)=r^{-1}r(x,y)=r^{-1}(\sigma_x(y),\tau_y(x))=(\hat{\sigma}_{\sigma_x(y)}\tau_y(x),\hat{\tau}_{\tau_y(x)}\sigma_x(y)).
\end{align*}
Hence 
\begin{align}
&\sigma_{\hat{\sigma}_x(y)}\hat{\tau}_y(x)=x\quad \Rightarrow\quad \hat{\tau}_y(x)=\sigma^{-1}_{\hat{\sigma}_x(y)}(x)\quad \Rightarrow\quad \sigma^{-1}_y(x)=\hat{\tau}_{\hat{\sigma}^{-1}_x(y)}(x)\label{rr:1}\\
&\tau_{\hat{\tau}_y(x)}\hat{\sigma}_x(y)=y\quad \Rightarrow\quad \hat{\sigma}_x(y)=\tau^{-1}_{\hat{\tau}_y(x)}(y)\quad \Rightarrow\quad \tau^{-1}_x(y)=\hat{\sigma}_{\hat{\tau}^{-1}_y(x)}(y)\label{rr:2}\\
&\hat{\sigma}_{\sigma_x(y)}\tau_y(x)=x\quad \Rightarrow\quad \tau_y(x)=\hat{\sigma}^{-1}_{\sigma_x(y)}(x)\quad \Rightarrow\quad \hat{\sigma}^{-1}_y(x)=\tau_{\sigma^{-1}_x(y)}(x)\label{rr:3}\\
&\hat{\tau}_{\tau_y(x)}\sigma_x(y)=y\quad \Rightarrow\quad \sigma_x(y)=\hat{\tau}^{-1}_{\tau_y(x)}(y)\quad \Rightarrow\quad \hat{\tau}^{-1}_x(y)=\sigma_{\tau^{-1}_y(x)}(y).\label{rr:4}
\end{align}
By definition, in an involutive solution we always have $\hat{\sigma}_x=\sigma_x$ and $\hat{\tau}_x=\tau_x$.

A solution $(X,\sigma,\tau)$ is called \emph{permutational}, if for every $x,y\in X$, $\sigma_x=\sigma_y$  and $\tau_x=\tau_y$. It is a \emph{projection} (or \emph{trivial}) solution if for every $x\in X$, $\sigma_x=\tau_x=\id$.

We say that a solution $(X,\sigma,\tau)$ satisfies  Condition ({\bf lri}) if, for each $x\in X$, the permutations $\sigma_x$ and $\tau_x$ are mutually inverse, i.e. 
\begin{align}
\tag{{\bf lri}}
\forall_{x\in X}\quad \sigma_x=\tau_x^{-1}.
\end{align} 

If $(X, \sigma,\tau)$ is a solution then directly by braid relation we obtain for $x,y,z\in X$:
\begin{align}
&\sigma_x\sigma_y=\sigma_{\sigma_x(y)}\sigma_{\tau_y(x)} \label{birack:1}\\
&\tau_{\sigma_{\tau_y(x)}(z)}\sigma_x(y)=\sigma_{\tau_{\sigma_y(z)}(x)}\tau_{z}(y) \label{birack:2}\\
&\tau_x\tau_y=\tau_{\tau_x(y)}\tau_{\sigma_y(x)} \label{birack:3}
\end{align}

To describe involutive solutions $(X,\sigma,\tau)$, the groups of the form $\langle \sigma_x:\ x\in X\rangle$ were investigated by many authors (see e.g. \cite{CJR}) and called \emph{IYB groups}. In the non-involutive cases various types of such kind of groups were introduced -- see e.g. \cite{B18, JPZ20, CJKAV}.
 In the paper we will focus on three of them.
The \emph{left permutation group} of a solution $(X,\sigma,\tau)$ is the permutation group generated by the permutations from the left-hand coordinate of~$r$, i.e. the group $\LMlt(X)=\langle \sigma_x:\ x\in X\rangle$. Similarly, one defines the \emph{right permutation group} of $(X,\sigma,\tau)$ as the permutation group $\RMlt(X)=\langle \tau_x:\ x\in X\rangle$. The permutation group $\Mlt(X)=\langle\sigma_x,\tau_y:x,y\in X\rangle$ generated by all translations $\sigma_x$ and $\tau_x$ is called the \emph{permutation group} of a solution.

Etingof, Schedler and Soloviev (\cite{ESS}) introduced, for each solution $(X,\sigma,\tau)$, its \emph{structure group} $G(X,r):=\langle X\mid x\circ y=\sigma_x(y)\circ\tau_y(x)\; \forall x,y\in X\rangle$. A solution is called \emph{injective} if the canonical map $X\to G(X,r);\; x\mapsto x$ is injective. Involutive solutions are always injective.

Let us recall that a bijection $\Phi\colon X\to X'$ is an \emph{isomorphism} of two solutions $(X,\sigma,\tau)$ and $(X',\sigma',\tau')$ if, for each $x\in X$,
\[
\Phi\sigma_x=\sigma'_{\Phi(x)}\Phi\quad {\rm and}\quad \Phi\tau_x=\tau'_{\Phi(x)}\Phi.
\]

\subsection*{Congruences}

Let $(X,\sigma,\tau)$ be a solution. An equivalence relation $\mathord{\asymp}\subseteq X\times X$ such that, for $x_1,x_2,y_1,y_2\in X$, 
\begin{align}\label{congr}
&x_1\asymp x_2\;\; {\rm and} \;\; y_1\asymp y_2\quad \Rightarrow\quad \sigma^{\varepsilon}_{x_1}(y_1)\asymp \sigma^{\varepsilon}_{x_2}(y_2)\quad {\rm and}\quad \tau^{\varepsilon}_{x_1}(y_1)\asymp \tau^{\varepsilon}_{x_2}(y_2),
\end{align}
where $\varepsilon\in \{-1,1\}$, is called a \emph{congruence} of the solution $(X,\sigma,\tau)$.
A congruence induces a quotient solution on its classes.

In \cite{ESS} Etingof, Schedler and Soloviev introduced, for each involutive solution $(X,\sigma,\tau)$, the equivalence relation $\sim$ on the set $X$: for each $x,y\in X$
\begin{align}\label{rel:sim}
x\sim y\quad \Leftrightarrow\quad \sigma_x=\sigma_y
\end{align}
and they showed that 
$\sim$ is a congruence of the solution.
In the case of non-involutive solution $(X,\sigma,\tau)$,
the equivalence relation $\sim$
need not to 
be a congruence.
But it is so if the solution is left distributive (see \cite[Theorem 3.4]{JPZ20}). The quotient solution $(X^{\sim},\sigma,\tau)$ is then called the {\em left retract} of~$X$ and denoted by $\mathrm{LRet}(X,\sigma,\tau)$.

Analogously to \eqref{rel:sim}, we can define the symmetrical relation
\begin{equation}
x\backsim y \quad \Leftrightarrow\quad \tau_x=\tau_y
\end{equation}
and this relation induces a solution on the quotient set $X^{\backsim}$ of every right distributive solution. The quotient solution $(X^{\backsim},\sigma,\tau)$ is called the {\em right retract} of~$(X,\sigma,\tau)$ and denoted by $\mathrm{RRet}(X,\sigma,\tau)$.
If a 
solution is involutive then $x\sim y$ if and only if $x\backsim y$ \cite[Proposition 2.2]{ESS}.

The intersection of the two relations here defined is
the relation
\begin{equation}
 x\approx y \quad \Leftrightarrow \quad {x\sim y} \wedge x\backsim y
 \quad \Leftrightarrow \quad {\sigma_x=\sigma_y} \wedge {\tau_x=\tau_y}.
\end{equation}

Lebed and Vendramin showed in \cite{LV} that  the relation $\approx$ is a congruence of injective solutions. 
In \cite{JPZ19} the authors together with Zamojska-Dzienio proved that the relation $\mathrel{\approx}$ induces a solution on the quotient set $X^{\mathrel{\approx}}$  for any solution $(X,\sigma,\tau)$.  A substantially shorter proof has recently appeared in \cite{CJKAV}.
\begin{de}\label{ret}
Let $(X,\sigma,\tau)$ be a solution. The quotient solution $\mathrm{Ret}(X,\sigma,\tau):=(X^{\approx},\sigma,\tau)$  with $\sigma_{x^{\approx}}(y^{\approx})=\sigma_x(y)^{\approx}$ and $\tau_{y^{\approx}}(x^{\approx})=\tau_y(x)^{\approx}$, for $x^{\approx},y^{\approx}\in X^{\approx}$  and $x\in x^{\approx},\; y\in y^{\approx}$, is called the \emph{retraction} solution of $(X,\sigma,\tau)$.
We say that a solution $(X,\sigma,\tau)$ is \emph{irretractable} if ${\rm Ret}(X,\sigma,\tau)=(X,\sigma,\tau)$, i.e.
$\approx$ is the trivial relation.
On the other hand, if there exists an integer~$k$ such that $\mathrm{Ret}^k(X,\sigma,\tau)$ has one element only then we say that
$(X,\sigma,\tau)$ has {\em multipermutation level $k$}.
\end{de}

\section{$2$-reductive solutions}\label{sec:2red}

In this section we reprove the results of \cite{JPZ20a} concerning 2-reductive solutions.
A question is how to naturally generalize
the notion for non-involutive solutions;
a fundamental property of involutive $2$-reductive solutions is that the orbits of the permutation
group lie within the classes of~$\sim$.
Since, for non-involutive solutions, the congruence
$\sim$ is generalized by~$\approx$, it is natural to
assume that the orbits of the permutation group
lie within the same classes of~$\approx$; this is encoded by Identities \eqref{eq:red1}--\eqref{eq:red4}.
It then turns out that the generic construction of involutive 2-reductive solutions straightforwardly generalizes for non-involutive ones.

\begin{de}\label{def:2red}
A solution $(X,\sigma, \tau)$ is called $2$-\emph{reductive} if, for every $x,y\in X$:
\begin{align}
& \sigma_{\sigma_x(y)}=\sigma_y, \label{eq:red1}\\
& \tau_{\tau_x(y)}=\tau_y, \label{eq:red2}\\
& \sigma_{\tau_x(y)}=\sigma_y, \label{eq:red3}\\
& \tau_{\sigma_x(y)}=\tau_y. \label{eq:red4}
\end{align}
\end{de}

\begin{example}
Let $(A,+)$ be an abelian group and $0\neq \alpha\in \mathrm{End}(A,+)$ be such that $\alpha^2=0$. By \cite[Lemma 8]{Sol}, $(A,\sigma,\tau)$ with $\sigma_x(y)=\alpha(x)+y$ and $\tau_y(x)=x+\alpha(y)$, for $x,y\in A$, is a solution. 
A~straightforward calculation shows that $(A,\sigma,\tau)$ is $2$-reductive.
By \cite[Theorem 3.1]{Sol} it is injective if and only if $2\alpha=0$. In this case it is involutive.
\end{example}

For $2$-reductive solution we also have:
\begin{align}\label{eq:more2red}
& \sigma_{\sigma^{-1}_x(y)}=\sigma_y,\quad\tau_{\tau^{-1}_x(y)}=\tau_y , \quad \tau_{\sigma^{-1}_x(y)}=\tau_y\quad{\rm and}\quad \sigma_{\tau^{-1}_x(y)}=\sigma_y.
\end{align}

In the case of involutive solutions~\cite{JPZ20a} the property of 2-reductivity was 
defined by \eqref{eq:red1} only. It is nevertheless easy to prove that all the properties
\eqref{eq:red1}--\eqref{eq:more2red} are equivalent for involutive solutions.

For an involutive solution $(X,\sigma,\tau)$,  Gateva-Ivanova considered in \cite[Definition 4.3]{GI18} a condition saying
\begin{align*}
\tag{$\ast$} \forall x\in X\quad \exists y\in X\quad  \sigma_y(x)=x.
\end{align*}
It is evident, that each square free solution satisfies Condition $(\ast)$. On the other hand solutions without fixed points are examples of ones which do not satisfy this condition.

\begin{fact}\cite[Proposition 8.2]{GIC12}, \cite[Proposition 4.7]{GI18}
If an involutive solution satisfies Condition~$(\ast)$ then it is a multipermutation solution of level $2$ if and only if it is $2$-reductive.
\end{fact}
For non-involutive solutions we have a similar result.

\begin{proposition}\label{prop:star}
  Let $(X,\sigma,\tau)$ be a solution satisfying the following two properties:
  \begin{align}
    \forall x\in X \ \exists y\in X\quad \sigma_y(x)&=x,\label{prop:star:1}\\
    \forall x\in X \ \exists y\in X\quad \tau_y(x)&=x.\label{prop:star:2}
  \end{align}
  Then $(X,\sigma,\tau)$ is a multipermutation solution of level at most~$2$ if and only if it is $2$-reductive.
\end{proposition}

\begin{proof}
  A solution is of multipermutation level~$1$ if and only if $\sigma_x(z)=\sigma_y(z)$ and $\tau_x(z)=\tau_y(z)$, for all $x,y,z\in X$.
  Hence a solution is of multipermutation level at most~$2$ if and only if $\sigma_x(z)\approx \sigma_y(z)$ and $\tau_x(z)\approx \tau_y(z)$, for all $x,y,z\in X$.
  This can be rewritten into four identities:
  \begin{align*}
    \sigma_{\sigma_x(z)}&=\sigma_{\sigma_y(z)}, & \sigma_{\tau_x(z)}&=\sigma_{\tau_y(z)},\\
    \tau_{\sigma_x(z)}&=\tau_{\sigma_y(z)}, & \tau_{\tau_x(z)}&=\tau_{\tau_y(z)}.
  \end{align*}
It is therefore clear that a $2$-reductive solution is always of multipermutation level~$2$. On the other hand, if a solution is of multipermutation level~$2$ then, in each of the identities, we can pick a suitable $y$ such that $\sigma_y(z)=z$ or $\tau_y(z)=z$, respectively which transforms the identities into $2$-reductivity.
\end{proof}

\begin{prop}\label{th:disabel}
Let $(X,\sigma,\tau)$ be a solution. Then
\begin{enumerate}
\item [(i)] if the solution satisfies \eqref{eq:red1} and \eqref{eq:red3} then the left permutation group $\LMlt(X)$ is abelian;
\item [(ii)] if the solution satisfies \eqref{eq:red2} and \eqref{eq:red4} then the right permutation group $\RMlt(X)$ is abelian;
\item [(iii)] if the solution is~$2$-reductive
then the permutation group $\Mlt(X)$ is abelian.
\end{enumerate}
\end{prop}
\begin{proof}
Let $(X,\sigma,\tau)$ be a solution.

\begin{align*}
\text{(i) }&\sigma_x\sigma_y\stackrel{\eqref{birack:1}}=\sigma_{\sigma_x(y)}\sigma_{\tau_y(x)} \stackrel{\eqref{eq:red1},\eqref{eq:red3}}=\sigma_y\sigma_x;\\
\text{(ii) }&\tau_x\tau_y\stackrel{\eqref{birack:3}}=\tau_{\tau_x(y)}\tau_{\sigma_y(x)} \stackrel{\eqref{eq:red2},\eqref{eq:red4}}=\tau_y\tau_x; \\
\text{(iii) }&\tau_z\sigma_x(y)\stackrel{\eqref{eq:red4}}=\tau_{\sigma_x(z)}\sigma_x(y) \stackrel{\eqref{eq:red3}}=
\tau_{\sigma_{\tau_y(x)}(z)}\sigma_x(y)
\stackrel{\eqref{birack:2}}=
\sigma_{\tau_{\sigma_y(z)}(x)}\tau_z(y)\stackrel{\eqref{eq:red4}}=
\sigma_{\tau_z(x)}\tau_z(y)\stackrel{\eqref{eq:red3}}=
\sigma_x\tau_z(y). 
\end{align*}
\end{proof}

Gateva-Ivanova showed \cite{GI18} that $2$-reductive involutive solutions always satisfy Condition ({\bf lri}). 
Actually, involutive solutions are the only $2$-reductive ones satisfying ({\bf lri}).

\begin{lemma}
$2$-reductive solution satisfies $\sigma_x=\tau_x^{-1}$ if and only if it is involutive.
\end{lemma}
\begin{proof}
Let $(X,\sigma,\tau)$ be $2$-reductive solution which satisfies Condition ({\bf lri}). Then, for each $x\in X$, $\sigma_x=\tau_x^{-1}$. Hence by \eqref{eq:more2red}, for every $x,y\in X$, we obtain
\begin{align*}
&\tau_y(x)=\sigma^{-1}_y(x)\stackrel{\eqref{eq:more2red}}=
\sigma^{-1}_{\tau_x^{-1}(y)}(x)
\stackrel{\eqref{rr:4}}{=}\hat\tau_x(y)
\quad{\rm and}\\
&\sigma_x(y)=\tau^{-1}_x(y)
\stackrel{\eqref{eq:more2red}}=
\tau^{-1}_{\sigma^{-1}_y(x)}(y)
\stackrel{\eqref{rr:3}}{=}\hat\sigma_x(y).
&& \qedhere
\end{align*}
\end{proof}
\noindent
In Section~\ref{sec:distr} we shall encounter similar identities, namely $\sigma_x=\hat\tau_x^{-1}$ and $\tau_x=\hat\sigma_x^{-1}$.
\vskip 2mm

Now we shall present a construction of, not necessarily involutive, $2$-reductive solutions based on  abelian groups and we will obtain an example of a family of $2$-reductive solutions. The idea of this combinatorial construction originates from involutive case (see \cite[Theorem 7.8]{JPZ20a}). The same construction for involutive solutions was described by Rump in the language of category theory under the name \emph{transvection torsor}  (see \cite[Definition 2]{Rump22}).

\begin{theorem}\label{thm:affmesh}
Let $I$ be a non-empty set, $(A_i)_{i\in I}$ be a family of abelian groups over $I$, $\bigcup_{i\in I} A_i$ be the disjoint union of the sets $A_i$, $c_{i,j},d_{i,j}\in A_j$, for $i,j\in I$, be some constants. Then $(\bigcup_{i\in I} A_i,\sigma,\tau)$, where  for $x\in A_i$, $y\in A_j$,
\begin{equation}\label{eq:7.8}
\sigma_x(y)=y+c_{i,j}\quad {\rm and} \quad \tau_y(x)= x+d_{j,i},
\end{equation}
is a $2$-reductive solution.
\end{theorem}

\begin{proof}
Clearly, for each  $x\in A_i$ and $y\in A_j$, the mappings $\sigma_x$ and $\tau_y$ are bijections with 
\begin{equation*}
\sigma^{-1}_x(y)=y-c_{i,j}\quad {\rm and} \quad \tau^{-1}_y(x)= x-d_{j,i}.
\end{equation*}
Moreover, for $z\in A_k$
\begin{align*}
\sigma_x\sigma_y(z)&=\sigma_x(z+c_{j,k})=z+c_{j,k}+c_{i,k}=\\
&=z+c_{i,k}+c_{j,k}=\sigma_{(y+c_{i,j})}(z+c_{i,k})=\sigma_{(y+c_{i,j})}\sigma_{(x+d_{j,i})}(z)=\sigma_{\sigma_x(y)}\sigma_{\tau_y(x)}(z);\\
\tau_{\sigma_{\tau_y(x)}(z)}\sigma_x(y)&=
\tau_{\sigma_{(x+d_{j,i})}(z)}(y+c_{i,j})=
\tau_{(z+c_{i,k})}(y+c_{i,j})=y+c_{i,j}+d_{k,j}=\\
&=y+d_{k,j}+c_{i,j}=\sigma_{(x+d_{k,i})}(y+d_{k,j})=\sigma_{\tau_{(z+c_{j,k})}(x)}(y+d_{k,j})=
\sigma_{\tau_{\sigma_y(z)}(x)}\tau_{z}(y);\\
\tau_x\tau_y(z)&=\tau_x(z+d_{j,k})=z+d_{j,k}+d_{i,k}=\\
&=z+d_{i,k}+d_{j,k}=\tau_{(y+d_{i,j})}(z+d_{i,k})=\tau_{(y+d_{i,j})}\tau_{(x+c_{j,i})}(z)=
\tau_{\tau_x(y)}\tau_{\sigma_y(x)}(z),
\end{align*}
which shows that \eqref{birack:1}--\eqref{birack:3} are satisfied.

Further,
\begin{align*}
& \sigma_{\sigma_x(y)}(z)=\sigma_{(y+c_{i,j})}(z)=z+c_{j,k}=\sigma_y(z), \\
& \tau_{\tau_x(y)}(z)=\tau_{(y+d_{i,j})}(z)=z+d_{j,k}=\tau_y(z), \\
& \sigma_{\tau_x(y)}(z)=\sigma_{(y+d_{i,j})}(z)=z+c_{j,k}=\sigma_y(z),\\
& \tau_{\sigma_x(y)}(z)=\tau_{(y+c_{i,j})}(z)=z+d_{j,k}=\tau_y(z),
\end{align*}
which justifies $2$-reductivity.
\end{proof}

\begin{example}\label{ex:affmesh}
Let $I$ be a (finite or infinite) index set
and let $A_i$, for $i\in I$, be cyclic groups.
Let $(c_{i,j})_{i,j\in I}$ and $(d_{i,j})_{i,j\in I}$ be constants
such that $c_{i,j},d_{i,j}\in A_j$, for all~$i,j\in I$.
Then $(\bigcup_{i\in I} A_i,\sigma,\tau)$, with $\sigma$ and $\tau$
defined in \eqref{eq:7.8}, is a $2$-reductive solution. If for each~$j\in I$, there exists at least one~$i\in I$,
such that $c_{i,j}$ or $d_{i,j}$ is a generator of the group $A_j$ then orbits of the action of 
$\Mlt(X)$ equal to $A_j$.
\end{example}
In general, since $\sigma_x\tau_y(z)=z+c_{i,k}+d_{j,k}\in A_k$, the group $\Mlt(X)$ acts transitively on $A_k$ if and only if $A_k=\left<\{c_{i,k},d_{j,k}\mid i\in I\}\right>$, for every $k\in I$. Hence if we assume that 
\begin{align}\label{a:mesh}
A_j=\left<\{c_{i,j},d_{i,j}\mid i\in I\}\right>,\; {\rm for \;every} \;j\in I,
\end{align}
 then the solution has orbits of the action of $\Mlt(X)$ equal to $A_j$, $j\in I$ and each orbit is a permutational solution. 

We will denote the solution satisfying \eqref{a:mesh} by $\mathcal A=((A_i)_{i\in I},\,(c_{i,j})_{i,j\in I},\,(d_{i,j})_{i,j\in I})$ and call it \emph{the disjoint union, over a set $I$, of abelian groups}.

\begin{theorem}\label{th:2red}
A solution $(X,\sigma,\tau)$ is $2$-reductive if and only if it is a disjoint union, over a set $I$, of abelian groups. The orbits of the action of $\Mlt(X)$ coincide with the groups.
\end{theorem}

\begin{proof}
By comment after Example \ref{ex:affmesh}, the disjoint union, over a set $I$, of abelian groups is $2$-reductive solution with orbits of the action of $\Mlt(X)$ equal to $A_j$, $j\in I$.

Now let $(X,\sigma,\tau)$ be a 2-reductive solution, and choose a transversal $E$ to the orbit decomposition of the action of $\Mlt(X)$.
By Theorem \ref{th:disabel}, the group $\Mlt(X)$ is abelian. Hence, for every $e\in E$, the orbit $Xe=\{\alpha(e)\mid \alpha \in \Mlt(X)\}$ 
can be endowed with an abelian group structure $(Xe,+,-,e)$ defined by $\alpha(e)+\beta(e)=\alpha\beta(e)$ and $-\alpha(e)=\alpha^{-1}(e)$, for $\alpha,\beta\in \Mlt(X)$.

Let, for every $e,f\in E$,
\[
c_{e,f}:=\sigma_e(f)\in Xf\; {\rm and}\; d_{e,f}:=\tau_e(f)\in Xf.
\]

Since $\Mlt(X)$ is abelian, and the solution is $2$-reductive then, for each $\alpha\in \Mlt(X)$, we have $\sigma_{\alpha(e)}=\sigma_e$ and $\tau_{\alpha(e)}=\tau_e$. This implies that the set 
\begin{align*}
&\{c_{e,f}, d_{e,f}\mid e\in E\}=
\{\sigma_e(f),\tau_e(f)\mid e\in E\}=\\
&\{\sigma_{\alpha(e)}(f),\tau_{\alpha(e)}(f)\mid \alpha\in \Mlt(X), e\in E\}=
\{\sigma_x(f),\tau_x(f)\mid x\in X\}
\end{align*}
generates the group $(Xf,+,-,f)$.
This shows that 
the disjoint union of abelian groups over a set~$E$,
$((Xe)_{e\in E},\,(c_{e,f})_{e,f\in E},\, (d_{e,f})_{e,f\in E})$
has exactly the same orbits as the solution~$X$.

Finally, let $x=\alpha(e)\in Xe$ and $y=\beta(f)\in Xf$ with $\alpha, \beta\in \Mlt(X)$. Therefore we obtain
\begin{align*}
&\sigma_x(y)=\sigma_{\alpha(e)}\beta(f)=\sigma_e(f)+\beta(f)=c_{e,f}+y,\; {\rm and}\\
&
\tau_y(x)=\tau_{\beta(f)}\alpha(e)=\tau_f(e)+\alpha(e)=d_{f,e}+x.
\end{align*}
So we verified that the disjoint union of abelian groups $((Xe)_{e\in E},\,(c_{e,f})_{e,f\in E},\, (d_{e,f})_{e,f\in E})$ yields the original solution $(X,\sigma,\tau)$.
\end{proof}

Note that the disjoint union of abelian groups is 
square free if and only if $c_{i,i}=d_{i,i}=0$, for each $i\in I$ and it is involutive if and only if $d_{i,j}=-c_{i,j}$ (see
\cite{JPZ20a}).

\begin{prop}\label{prop:inv_2red}
 Let $(X,\sigma,\tau)$ be a disjoint union of
 $\mathcal A=((A_i)_{i\in I},\,(c_{i,j})_{i,j\in I},\,(d_{i,j})_{i,j\in I})$.
 Then $(X,\hat\sigma,\hat\tau)$ is the disjoint union of
 $\mathcal A=((A_i)_{i\in I},\,(-d_{i,j})_{i,j\in I},\,(-c_{i,j})_{i,j\in I})$.
\end{prop}

\begin{proof}
 According to \eqref{rr:3}, $\hat\sigma_x(y)=\tau_{\sigma_y^{-1}(x)}^{-1}(y)\stackrel{\eqref{eq:more2red}}{=}
 \tau_x^{-1}(y)\stackrel{\eqref{eq:7.8}}{=}y-d_{i,j}$.
 Analogously for $\hat\tau$.
\end{proof}

\begin{theorem}\label{th:iso}
Let  $\mathcal A=((A_i)_{i\in I},\,(c_{i,j})_{i,j\in I},\,(d_{i,j})_{i,j\in I})$ and  $\mathcal{A}^{'}=((A^{'}_i)_{i\in I},\,(c^{'}_{i,j})_{i,j\in I},\,(d^{'}_{i,j})_{i,j\in I})$ be two disjoint unions of abelian groups, over the same index set $I$.
Then the unions $\mathcal A$ and $\mathcal{A}^{'}$ are isomorphic $2$-reductive solutions $(\bigcup\limits_{i\in I} A_i,\sigma,\tau)$ and $(\bigcup\limits_{i\in I} A^{'}_i,\sigma^{'},\tau^{'})$ if and only if there is a bijection $\pi$ of the set $I$ and group isomorphisms $\psi_i\colon A_i\to A^{'}_{\pi (i)}$ such that 
\begin{align}\label{a:homol}
&\psi_j(c_{i,j})=c^{'}_{\pi (i),\pi (j)}\quad {\rm and}\quad \psi_j(d_{i,j})=d^{'}_{\pi (i),\pi (j)},
\end{align} for every $i,j\in I$.
\end{theorem}

\begin{proof}
The proof goes in similar way as the proof of \cite[Theorem 4.2]{JPSZ} for medial quandles in the case of 2-reductive ones.

\noindent
$(\Leftarrow)$
Let us define a mapping $\psi:\bigcup A_i\to\bigcup A_i'$ by $$\psi(x)=\psi_i(x),$$ for every $x\in A_i$. We will prove that $\psi$ is an isomorphism between the solutions. It is clearly a bijection. Let $x\in A_i$, $y\in A_j$. Using the fact that $\psi_j$ is a group homomorphism, we obtain
\begin{align*}
&\psi(\sigma_x(y))=\psi_j(\sigma_x(y))=
\psi_j(y)+\psi_j(c_{i,j}) \; {\rm and}\\
& \psi(\tau_x(y))=\psi_j(\tau_x(y))=
\psi_j(y)+\psi_j(d_{i,j}).
\end{align*}
On the other hand,
\begin{align*}
&\sigma^{'}_{\psi(x)}(\psi(y))=\sigma^{'}_{\psi_i(x)}(\psi_j(y))=
\psi_j(y)+c^{'}_{\pi(i),\pi(j)}\; {\rm and}\;\\
&\tau^{'}_{\psi(x)}(\psi(y))=\tau^{'}_{\psi_i(x)}(\psi_j(y))=
\psi_j(y)+d^{'}_{\pi(i),\pi(j)}.
\end{align*}
By \eqref{a:homol} we can see the two expressions are equal.

\noindent
$(\Rightarrow)$
Let $f$ be an isomorphism between the two disjoint unions of abelian groups. Since isomorphisms preserve orbits, there is a permutation $\pi$ of $I$ such that $f(A_i)=A_{\pi(i)}'$ for every $i\in I$. 

Let, for $j\in I$, $0_j\in A_j$ be the neutral element in the group $A_j$. Let define the mappings
\begin{equation*}
\psi_i\colon A_i\to f(A_i)=A_{\pi(i)}';\qquad \psi_i(x)=f(x)-f(0_i),
\end{equation*}
for every $i\in I$.  
Since $f$ is an isomorphism of solutions for $x\in A_i$ we have:
\begin{align*}
f(0_j)+c^{'}_{\pi(i),\pi(j)}=\sigma^{'}_{f(x)}(f(0_j))=f(\sigma_x(0_j))=f(0_j+c_{i,j})=f(c_{i,j})=\psi_j(c_{i,j})+f(0_j).
\end{align*}
This gives $\psi_j(c_{i,j})=c^{'}_{\pi(i),\pi(j)}$. Similarly,
\begin{align*}
f(0_j)+d^{'}_{\pi(i),\pi(j)}=\tau^{'}_{f(x)}(f(0_j))=f(\tau_x(0_j))=f(0_j+d_{i,j})=f(d_{i,j})=\psi_j(d_{i,j})+f(0_j),
\end{align*}
and $\psi_j(d_{i,j})=d^{'}_{\pi(i),\pi(j)}$.

To verify that the mappings $\psi_j$ are automorphisms of groups, let  for $x\in A_i$ and $y\in A_j$, consider:
\begin{align*}
&f(\sigma_x(y))=f(y+c_{i,j})=\psi_j(y+c_{i,j})+f(0_j)\; {\rm and}\\
&f(\tau_x(y))=f(y+d_{i,j})=\psi_j(y+d_{i,j})+f(0_j).
\end{align*}
On the other hand,
\begin{align*}
&f(\sigma_x(y))=\sigma^{'}_{f(x)}(f(y))=c_{\pi(i),\pi(j)}'+ f(y)=
\psi_j(c_{i,j})+\psi_j(y)+f(0_j)\; {\rm and}\\
&f(\tau_x(y))=\tau^{'}_{f(x)}(f(y))=d_{\pi(i),\pi(j)}'+ f(y)=
\psi_j(d_{i,j})+\psi_j(y)+f(0_j)
\end{align*}
Cancelling $f(0_j)$ we obtain
\begin{equation}\label{eq:3}
\psi_j(c_{i,j}+y)=
\psi_j(c_{i,j})+\psi_j(y)\quad {\rm and}\quad \psi_j(d_{i,j}+y)=
\psi_j(d_{i,j})+\psi_j(y).
\end{equation}
for every $y\in A_j$.

By assumption, every group $A_j$ is generated by all the elements $c_{i,j}$ and $d_{i,j}$, $i\in I$.
Hence \eqref{eq:3} implies $\psi_j(x+y)=\psi_j(x)+\psi_j(y)$ for every $x,y\in A_j$, i.e., $\psi_j$ is an automorphism of groups.
\end{proof}

If $I$ is a finite set we will usually display a disjoint union of abelian groups as a triple $((A_i)_{i\in I},C,D)$, where $C=\,(c_{i,j})_{i,j\in I}$ and $D=\,(d_{i,j})_{i,j\in I}$ are $\left|I\right|\times \left|I\right|$ matrices of constants.
\vskip 2mm

We can construct all $2$-reductive solutions of~size~$n$ using
the following algorithm:
\begin{algorithm}
Outputs all $2$-reductive solutions of size~$n$:
\begin{enumerate}
 \item For all partitionings $n=n_1+n_2+\cdots +n_k$ do (2)--(4).
 \item For all abelian groups $A_1$, \dots ,$A_k$ of size $|A_i|=n_i$ do (3)--(4).
 \item For all constants $c_{i,j},d_{i,j}\in A_j$, $1\leq i,j\leq k$, do (4).
 \item If, for all $1\leq j\leq k$, we have $A_j=\langle \{c_{i,j},d_{i,j}\mid 1\leq i\leq k\}\rangle$
 then construct a solution $(\bigcup A_i,\sigma,\tau)$ using \eqref{eq:7.8}.
\end{enumerate}
\end{algorithm}
When all solutions are constructed, we can get rid of isomorphic
copies using Theorem~\ref{th:iso}.

\begin{example}
Up to isomorphism, there are exactly fourteen $2$-reductive solutions of size 3. They are the following disjoint unions of abelian groups:
\begin{itemize}
    \item One orbit: $(\Z_3,(1),(1))$, $(\Z_3,(0),(1))$, $(\Z_3,(1),(0))$.

        \item Two orbits: $(\Z_2\cup\Z_1,\left(\begin{smallmatrix}1&0\\0&0\end{smallmatrix}\right), \left(\begin{smallmatrix}0&0\\0&0\end{smallmatrix}\right))$,         
        $(\Z_2\cup\Z_1,\left(\begin{smallmatrix}1&0\\0&0\end{smallmatrix}\right), \left(\begin{smallmatrix}1&0\\0&0\end{smallmatrix}\right))$,        
        $(\Z_2\cup\Z_1,\left(\begin{smallmatrix}1&0\\0&0\end{smallmatrix}\right), \left(\begin{smallmatrix}0&0\\1&0\end{smallmatrix}\right))$,\\        
        $(\Z_2\cup\Z_1,\left(\begin{smallmatrix}1&0\\0&0\end{smallmatrix}\right), \left(\begin{smallmatrix}1&0\\1&0\end{smallmatrix}\right))$, $(\Z_2\cup\Z_1,\left(\begin{smallmatrix}0&0\\1&0\end{smallmatrix}\right),\left(\begin{smallmatrix}0&0\\0&0\end{smallmatrix}\right))$, $(\Z_2\cup\Z_1,\left(\begin{smallmatrix}0&0\\1&0\end{smallmatrix}\right),\left(\begin{smallmatrix}1&0\\0&0\end{smallmatrix}\right))$, $(\Z_2\cup\Z_1,\left(\begin{smallmatrix}0&0\\1&0\end{smallmatrix}\right),\left(\begin{smallmatrix}0&0\\1&0\end{smallmatrix}\right))$,\\        
$(\Z_2\cup\Z_1,\left(\begin{smallmatrix}0&0\\1&0\end{smallmatrix}\right),\left(\begin{smallmatrix}1&0\\1&0\end{smallmatrix}\right))$, $(\Z_2\cup\Z_1,\left(\begin{smallmatrix}1&0\\1&0\end{smallmatrix}\right), \left(\begin{smallmatrix}0&0\\0&0\end{smallmatrix}\right))$, $(\Z_2\cup\Z_1,\left(\begin{smallmatrix}1&0\\1&0\end{smallmatrix}\right), \left(\begin{smallmatrix}1&0\\0&0\end{smallmatrix}\right))$, $(\Z_2\cup\Z_1,\left(\begin{smallmatrix}1&0\\1&0\end{smallmatrix}\right), \left(\begin{smallmatrix}0&0\\1&0\end{smallmatrix}\right))$,\\
$(\Z_2\cup\Z_1,\left(\begin{smallmatrix}1&0\\1&0\end{smallmatrix}\right), \left(\begin{smallmatrix}1&0\\1&0\end{smallmatrix}\right))$.

\item Three orbits: $(\Z_1\cup\Z_1\cup\Z_1,\,\left(\begin{smallmatrix}0&0&0\\0&0&0\\0&0&0\end{smallmatrix}\right),\left(\begin{smallmatrix}0&0&0\\0&0&0\\0&0&0\end{smallmatrix}\right))$.
\end{itemize}
Five of them are involutive and three are square free.
Furthermore, there are exactly ninety six $2$-reductive solutions of size $4$:
\begin{itemize}
    \item 3 with one orbit $\Z_4$,
        \item 20 with two orbits:
        $\Z_3$ and $\Z_1$ and 42 with two orbits 
        $\Z_2$ and $\Z_2$,
         \item 30 with three orbits: $\Z_2$, $\Z_1$ and $\Z_1$,
        \item 1 with four orbits each equal to $\Z_1$.
\end{itemize}
\end{example}

The representation of $2$-reductive solutions as a disjoint union of abelian groups allows one to quickly verify the conditions defined in Proposition \ref{prop:star}. 
\begin{rem}\label{rem:star}
Let a $2$-reductive solution $(X,\sigma,\tau)$ be  a disjoint union of abelian groups over a set $I$. Then $(X,\sigma,\tau)$ satisfies \eqref{prop:star:1} - \eqref{prop:star:2} if and only if
\[
\forall i\in I\quad\exists j\in I,\quad\text{such that}\quad c_{j,i}=d_{j,i}=0.
\]
\end{rem}

Injective solutions were investigated, among others,  by Soloviev in \cite{Sol}. He showed that some properties of such solutions are similar to that of involutive ones. He also gave some criterions to recognize injective solutions. In particular, he characterized affine injective solutions. 
Here we have presented
$2$-reductive solutions as disjoint unions of abelian groups. This presentation shows, in fact, that each their component is an affine solution. Hence, it allows us to formulate some conditions for the elements lying in diagonals in matrices of constants.

\begin{prop}\cite{Sol}
 Let $(X,\sigma,\tau)$ be a disjoint union of
 $\mathcal A=((A_i)_{i\in I},\,(c_{i,j})_{i,j\in I},\,(d_{i,j})_{i,j\in I})$. If $(X,\sigma,\tau)$ is injective then,
 for all $i\in I$, $c_{i,i}=-d_{i,i}$.
\end{prop}

\begin{proof}
 Let $x\in A_i$, for $i\in I$. Then, in the structure group $G(X,r)$ of the solution $(X,\sigma,\tau)$,
 \[ x\circ \sigma_{x}^{-1}(x)=\sigma_x\sigma_x^{-1}(x)\circ \tau_{\sigma_{x}^{-1}(x)}(x)\stackrel{\eqref{eq:more2red}}= x \circ \tau_x(x)
 \]
and, by cancellativity, $\sigma_x^{-1}(x)=\tau_x(x)$ and therefore $-c_{i,i}=d_{i,i}$, for all $i\in I$.
\end{proof}

\begin{proposition}
 Let $(X,\sigma,\tau)$ be a disjoint union of
 $\mathcal A=((A_i)_{i\in I},\,(c_{i,j})_{i,j\in I},\,(d_{i,j})_{i,j\in I})$. If $(X,\sigma,\tau)$ is injective then,
 for all $i,j\in I$,
 $o(c_{i,j}+d_{i,j})=o(c_{j,i}+d_{j,i})$.
\end{proposition}

\begin{proof}
 Let $x\in A_i$ and $y\in A_j$, for $i,j\in I$. In the structure group $G(X,r)$ we have
 \[x \circ y=\sigma_{x}(y) \circ \tau_{y}(x)=\sigma_{\sigma_{x}(y)}\tau_y(x) \circ \tau_{\tau_{y}(x)}\sigma_x(y)
 \stackrel{\eqref{eq:red1},\eqref{eq:red2}}=
 \sigma_{y}\tau_y(x) \circ\tau_{x}\sigma_x(y).
 \]
 Suppose that there exist $i,j\in I$ such that $k=o(c_{i,j}+d_{i,j})<o(c_{j,i}+d_{j,i})$. Then we obtain
 \[x \circ y=\sigma^k_y\tau_y^k(x) \circ\tau_x^k\sigma_x^k(y)=
 (x+kc_{j,i}+kd_{j,i}) \circ (y+kc_{i,j}+kd_{i,j})=
  (x+kc_{j,i}+kd_{j,i}) \circ y
 \]
 and we obtain $x=x+k(c_{j,i}+d_{j,i})\neq x$ in the structure group.
 Hence $(X,\sigma,\tau)$ is not an injective solution.
\end{proof}

\begin{example}
Let $(\{x,y,z\},\sigma, \tau)$ be the union of $(\Z_2\cup\Z_1,\left(\begin{smallmatrix}0&0\\0&0\end{smallmatrix}\right),\left(\begin{smallmatrix}0&0\\1&0\end{smallmatrix}\right)).$
This square-free solution is neither involutive, since $c_{2,1}\neq -d_{2,1}$,
nor injective since $c_{2,1}+d_{2,1}=1$ and $c_{1,2}+d_{1,2}=0$. Indeed, we have
$\sigma_x=\sigma_y=\sigma_z=\tau_x=\tau_y=\mathrm{id}_X$ and $\tau_z=(xy)$. Then
\[x \circ z= z \circ y=y \circ z\]
and the structure group of this solution is a free abelian group with two generators $\{x,z\}$.
\end{example}

\begin{problem}
For a finite $2$-reductive solution  $\mathcal A=((A_i)_{i\in I},\,C,\,D)$, find a necessary and sufficient condition on the matrices $C$ and $D$ so that the solution is injective.
\end{problem}

There were three equivalences defined on~$X$ in Section~\ref{sec:prelim}, namely $\sim$, $\backsim$ and $\approx$; only the last one is a congruence in general.
In the case of $2$-reductive solutions all of them are congruences.

\begin{prop}
 Let $(X,\sigma,\tau)$ be a $2$-reductive solution. Then $\sim$ and $\backsim$ are congruences of~$(X,\sigma,\tau)$.
\end{prop}

\begin{proof}
 Let $x\sim y$ and $a\sim b$. We need to prove four properties:
 \begin{align*}
  \sigma_{\sigma_x(a)}\stackrel{\eqref{eq:red1}}{=} \sigma_a=\sigma_b\stackrel{\eqref{eq:red1}}{=}\sigma_{\sigma_y(b)}&\quad\Rightarrow\quad \sigma_x(a)\sim\sigma_y(b),\\
  \sigma_{\sigma_x^{-1}(a)}\stackrel{\eqref{eq:more2red}}{=} \sigma_a=\sigma_b\stackrel{\eqref{eq:more2red}}{=}\sigma_{\sigma_y^{-1}(b)}&\quad\Rightarrow\quad \sigma_x^{-1}(a)\sim\sigma_y^{-1}(b),\\
  \sigma_{\tau_x(a)}\stackrel{\eqref{eq:red3}}{=} \sigma_a=\sigma_b\stackrel{\eqref{eq:red3}}{=}\sigma_{\tau_y(b)}&\quad\Rightarrow\quad \tau_x(a)\sim\tau_y(b),\\
  \sigma_{\tau_x^{-1}(a)}\stackrel{\eqref{eq:more2red}}{=} \sigma_a=\sigma_b\stackrel{\eqref{eq:more2red}}{=}\sigma_{\tau_y^{-1}(b)}&\quad\Rightarrow\quad \tau^{-1}_x(a)\sim\tau^{-1}_y(b).
 \end{align*}
The proof is analogous for $\backsim$.
\end{proof}

\begin{corollary}\label{lm:sf}
Let $(X,\sigma,\tau)$ be a $2$-reductive solution. Then the retraction solutions $\mathrm{Ret}(X,\sigma,\tau)$,
$\mathrm{LRet}(X,\sigma,\tau)$ and $\mathrm{RRet}(X,\sigma,\tau)$ are projection ones. 
\end{corollary}
\begin{proof}
Since the solution $(X,\sigma,\tau)$ is $2$-reductive, for $x,y\in X$ we immediately have that 
\[
\sigma_x(y)\approx y\quad {\rm and}\quad \tau_x(y)\approx y.
\]
This means that in $\mathrm{Ret}(X,\sigma,\tau)$, for $x^{\approx}\in X^{\approx}$, $\sigma_{x^{\approx}}=\tau_{x^{\approx}}=\id$. Since the relation $\approx$ is the intersection of $\sim$ and $\backsim$, the left and the right retracts are factors of $\mathrm{Ret}(X,\sigma,\tau)$.
\end{proof}

\begin{lemma}
Let a $2$-reductive solution $(X,\sigma,\tau)$ be  a disjoint union of abelian groups over a set $I$. For $x\in A_i$ and $y\in A_j$ one has $x\sim y$ if and only if $c_{i,k}=c_{j,k}$ and one has $x\backsim y$ if and only if $d_{i,k}=d_{j,k}$ for all $k\in I$.

\end{lemma}
\begin{proof}
Let $x\in A_i$ and $y\in A_j$. Then
\begin{multline*}
x\sim y 
 \quad\Leftrightarrow\quad \forall k\in I\quad\forall z\in A_k\quad z+c_{i,k}=\sigma_x(z)=\sigma_y(z)=z+c_{j,k}\; 
\quad \Leftrightarrow\quad \forall k\in I\quad c_{i,k}=c_{j,k}
\end{multline*}
and analogously for $\backsim$.
\end{proof}

\section{Skew left braces} \label{sec:skew}
The notion of left braces was introduced by Rump \cite{Rump07A} to investigate involutive solutions. Guarnieri and Vendramin \cite{GV} generalized the idea and defined so called \emph{skew left braces}. These structures provide the most accessible examples of solutions
and it is therefore natural to ask how the $2$-reductivity property translates in the context of skew left braces. In this section we mainly recall
properties of the skew left braces and its associated solutions.

\begin{de}\cite[Definition 1.1]{GV}\label{def:skew_left_brace}
 An algebra $(B,\cdot,\circ)$ is called a {\em skew left brace} if $(B,\cdot)$ and $(B,\circ)$ are groups and the operations satisfy, for all $a,b,c\in B$,
 \begin{equation}\label{lsb}
  a\circ(b\cdot c)=(a\circ b)\cdot a^{-1}\cdot (a\circ c).
 \end{equation}
 \end{de}
The neutral element in the group $(B,\cdot)$ and the neutral element in the group $(B,\circ)$ are equal and we will denote it by $1$.
Moreover, the inverse element to~$a$ in the group $(B,\cdot)$ shall be denoted by $a^{-1}$ and the inverse element to~$a$ in the group $(B,\circ)$  by $\bar a$. 
 
 A skew left brace is a left brace if the group $(B,\cdot)$ is abelian. In this case we say that the skew left brace is of \emph{abelian type}.
 
 \begin{lemma}\cite[Lemma 1.7]{GV}\label{lm:GV1.7}
 Let $(B,\cdot,\circ)$ be a skew left brace. Then the following hold for all $a,b,c\in B$:
 \begin{align}
 &a\circ(b^{-1}\cdot c)=a\cdot(a\circ b)^{-1}\cdot(a\circ c),\label{eq:1.7(2)}\\
 &a\circ(b\cdot c^{-1})=(a\circ b)\cdot(a\circ c)^{-1}\cdot a.\label{eq:1.7(3)}
 \end{align}
 \end{lemma}
 
A skew left brace is said to be a \emph{two-sided skew brace} if 
\begin{align}\label{eq:twosided}
(a\cdot b)\circ c=(a\circ c)\cdot c^{-1}\cdot (b\circ c),
\end{align}
holds for all $a,b,c\in B$.
For any skew left brace $(B,\cdot,\circ)$ there is so called \emph{opposite skew left brace} $(B,\cdot_{op},\circ)$, with $a\cdot_{op}b=b\cdot a$ (see \cite{KT}). 

\begin{remark}\cite[Remark 1.8, Proposition 1.9]{GV}\label{prop:lambda-hom}
Let $(B,\cdot,\circ)$ be a skew left brace. For each $a\in B$, the mapping
\[
\lambda_a\colon B\to B; \quad \lambda_a(b)=a^{-1}\cdot(a\circ b),
\]
is an automorphism of the group $(B,\cdot)$  with the inverse defined by $\lambda_a^{-1}(b)=\bar a\circ(a\cdot b)$.
\end{remark}

Similarly as for a left brace, one can define the \emph{associated solution} to a skew left brace.
\begin{theorem} \cite[Theorem 3.1]{GV}\label{Th:GV}
Let $(B,\cdot,\circ)$ be a skew left  
brace. Then $(B,\lambda,\rho)$, with 
\begin{align*}
&\rho_y(x):=\lambda^{-1}_{\lambda_x(y)}((x\circ y)^{-1}\cdot x\cdot(x\circ y))=\overline{x^{-1}\cdot (x\circ y)}\circ x\circ y=\overline{\lambda_x(y)}\circ x\circ y,\quad for\; x,y\in B,
\end{align*}
is a solution. It is involutive if and only 
the skew left brace $(B,\cdot,\circ)$ is of abelian type.
\end{theorem}
\noindent

Note that solutions associated to skew left braces satisfy the conditions \eqref{prop:star:1}--\eqref{prop:star:2}.

\smallskip

By results of Koch and Truman~\cite{KT}, for a skew left brace $(B,\cdot,\circ)$,
the solution associated to the opposite skew left brace is the inverse solution to the solution associated to $(B,\cdot,\circ)$.

\begin{theorem}\cite[Theorem 4.1]{KT}\label{thm:Koch}
Let $(B,\lambda,\rho)$ be the solution associated to a skew left brace $(B,\cdot,\circ)$. Then its inverse solution $(B,\hat\lambda,\hat\rho)$ is the solution associated to the opposite skew left brace $(B,\cdot_{op},\circ)$. 

In particular, for $x,y\in B$
\begin{align}
\hat\lambda_x(y)&=(x\circ y)\cdot x^{-1}, \quad {\rm and}\\
\hat\rho_y(x)&=\overline{(x\circ y)\cdot x^{-1}}\circ x\circ y.
\end{align}
\end{theorem}

\begin{example}\label{exm:trisol}
Let $(B,\cdot)$ be a group. Then $Triv(B,\cdot)=(B,\cdot,\cdot)$ is the \emph{trivial skew left brace} on $(B,\cdot)$. The associated solution $(B,\lambda,\rho)$ is of the form 
: $\lambda_a(b)=b$ and $\rho_b(a)=b^{-1}\cdot a\cdot b$.
\end{example}

\begin{example}\label{exm:almtrisol}
Let $(B,\cdot)$ be a group.
The associated solution to an \emph{almost trivial} skew left brace $(B,\cdot,\cdot_{op})$ is:
$\lambda_a(b)=a^{-1}\cdot b \cdot a$ and $\rho_b(a)=a$. Moreover, $\hat{\lambda}_a(b)=b$ and $\hat{\rho}_b(a)=b\cdot a\cdot b^{-1}$.\\
In particular, if a group $(B,\cdot)$ is abelian, then $(B,\lambda,\rho)$ is a projection solution.
\end{example}

\begin{remark}
Let~$(B,\cdot,\cdot_{op})$ be an almost trivial skew left brace and let $(B,\lambda,\rho)$ its associated solution. Then \eqref{eq:red2}, \eqref{eq:red3} and \eqref{eq:red4} are trivially satisfied and \eqref{eq:red1} is equivalent to
\[\lambda_{\lambda_x(y)}(z)=\lambda_y(z) \quad\Leftrightarrow\quad x^{-1}y^{-1}xzx^{-1}yx=y^{-1}zy \quad\Leftrightarrow\quad [[x,y],z]=1\]
and therefore the solution is $2$-reductive if and only if the group $(B,\cdot)$ is nilpotent of class~$2$.
\end{remark}

\begin{example}\label{exm:prodsol}
Let $(G,\cdot)$ be a non-abelian group. The solution $(G\times G,\lambda,\rho)$ associated with the product $(G,\cdot,\cdot)\times (G,\cdot,\cdot_{opp})$ of trivial and almost trivial skew left braces is, for $(x,y),(u,w)\in G\times G$ the following:
\begin{align*}
\lambda_{(x,y)}((u,w))=(u,y^{-1}\cdot w\cdot y)
\quad {\rm and}\quad \rho_{(u,w)}((x,y))=(u^{-1}\cdot x\cdot u,y).
\end{align*}
\end{example}
\begin{example}\label{exm:Z2n}
Let $n\in \mathbb{N}$ be an odd natural number and let $(\Z_{2n},\cdot,+_{2n})$ be a skew left brace such that for $x,y\in \Z_{2n}$
\begin{align*}
&x\cdot y=x+(-1)^{x}y\mod 2n\quad {\rm and}\quad 
x^{-1}=(-1)^{x+1}x\mod 2n.
\end{align*}
The solution associated with this skew left brace is the following:
\begin{align*}
\lambda_x(y)=\lambda_x^{-1}(y)=(-1)^xy\mod 2n\quad {\rm and}\quad \rho_x(y)=(-1)^{y+1}x+x+y\mod 2n.
\end{align*}
\end{example}

All these examples above were actually examples of bi-skew left braces.

\begin{de}\cite[Definition 2]{Chi}\label{def:SLB}
A skew left brace $(B,\cdot,\circ)$ is a \emph{bi-skew left brace} if $(B,\circ,\cdot)$ is a skew left brace as well.
\end{de}
Bi-skew left braces were introduced by Childs in \cite{Chi}, studied by Caranti in \cite{C} and also recently by Bardakov et al.  in \cite{BNY} under the name \emph{symmetric} skew left braces. Caranti (see also \cite[Proposition 5.2]{BNY}) presented the following equivalence. 
\begin{lemma}\cite[Lemma 3.1]{C}\label{lem:Car}
Let $(B,\cdot,\circ)$ be a skew left brace. 
The following conditions are equivalent:
\begin{enumerate}
\item [(i)]  $(B,\cdot,\circ)$ is a bi-skew left brace,
\item[(ii)]$
\lambda\colon B\to B; \quad a\mapsto \lambda_a
$
is an anti-homomorphism of groups $(B,\cdot)$ and $\aut{B,\cdot}$ that means $\lambda_{a\cdot b}=\lambda_b\lambda_a$, for all $a,b\in B$.
\end{enumerate}

\end{lemma}

Bi-skew left braces and their associated solutions were also studied by  
 Stefanello and Trappeniers in \cite{ST} who gave the following characterization:
\begin{thm}\cite[Theorem 5.2]{ST}\label{th:sbb}
Let $(B,\cdot,\circ)$ be a skew left brace. 
Then $(B,\cdot,\circ)$ is a bi-skew left brace if and only if its associated solution $(B,\lambda,\rho)$ satisfies
\begin{align}\label{eq:solbi} 
\lambda_{\hat{\lambda}_x(y)}=\lambda_y,\quad for\; all \;x,y\in B.
\end{align}
\end{thm}

This property is visually similar to \eqref{eq:red1} 
but it is actually equivalent to
a different identity defining 2-reductivity.

\begin{proposition}\label{prop:blr}
Let $(X,\sigma,\tau)$ be a solution. Then $(X,\sigma,\tau)$ satisfies  $\sigma_{\hat\sigma_x(y)}=\sigma_y$, for $x,y\in X$, if and only if it satisfies \eqref{eq:red3}.
\end{proposition}

\begin{proof}
Let $(X,\sigma,\tau)$ satisfy~\eqref{eq:solbi}. Then, for $x,y\in X$, we obtain 
\begin{align*}
\sigma_x\stackrel{\eqref{rr:3}}=\sigma_{\hat{\sigma}_{\sigma_x(y)}\tau_y(x)}\stackrel{\eqref{eq:solbi}}= \sigma_{\tau_y(x)}.
\end{align*}
On the other hand, for \eqref{eq:red3}, we have
\begin{align*}
\sigma_{\hat{\sigma}^{-1}_y(x)}\stackrel{\eqref{rr:3}}=\sigma_{\tau_{\sigma^{-1}_x(y)}(x)}\stackrel{\eqref{eq:red3}}=\sigma_x,
\end{align*}
and replacing $x$ by $\hat\sigma_y(x)$ completes the proof.
\end{proof}

\begin{corollary}
Each $2$-reductive solution satisfies \eqref{eq:solbi} and an involutive solution satisfying \eqref{eq:solbi} is $2$-reductive.
\end{corollary}

Analogously we can say that the identity
\begin{align}
\tau_{\hat{\tau}_x(y)}=\tau_y.\label{eq:solbir}
\end{align}
is equivalent to the identity~\eqref{eq:red4}.

\begin{example}\cite[Example 2.6]{JPZ20}
Let $(G,\cdot)$ be a non elementary abelian $2$-group such that for each $x\in G$, $x^2\in Z(G)$. Let us define a solution $(G,\sigma,\tau)$ as follows:
\begin{align*}
&\sigma_x(y)=xy^{-1}x^{-1}\quad {\rm and}\quad 
\tau_y(x)=xy^{-1}.
\end{align*}
This solution, with $\hat{\sigma}_x(y)=y^3$ and 
$\hat{\tau}_y(x)=x^{-1}y^{-1}x$, satisfies \eqref{eq:solbi} but not \eqref{eq:solbir}.
\end{example}

Let $(B,\cdot,\circ)$ be a skew left brace. 
\begin{de}
A subset $I\subseteq B$ is an \emph{ideal} of $(B,\cdot,\circ)$ if it is a normal subgroup of $(B,\cdot)$, a normal subgroup of $(B,\circ)$ and  $\lambda_a(I)\subseteq I$, for all $a\in B$. 
\end{de}
\begin{de}
The \emph{socle} of $(B,\cdot,\circ)$ is the ideal
\[
\Soc(B)=\{a\in B\colon a\circ b=a\cdot b
=b\cdot a,\; 
{\rm for\; all}\; b\in B\}=\ker{\lambda}\cap Z(B,\cdot),
\]
where $Z(B,\cdot)$ denotes the center of $(B,\cdot)$.
\end{de}

\begin{example}
For an almost trivial skew left brace $(B,\cdot,\cdot_{op})$, $\Soc(B)=Z(B,\cdot)$.
\end{example}

\begin{lemma}\cite[Lemma 3.3]{GV}\label{lem:cosets}
For a skew left brace $(B,\cdot,\circ)$ and an ideal $I\subseteq B$, for all $a\in B$, $a\cdot I=a\circ I$.
\end{lemma}

\begin{lemma}\cite[Lemma 1.10]{CSV}\label{lm:CSV}
Let $a\in \Soc(B)$. Then for all $b\in B$,
\[
\lambda_b(a)=b^{-1}\cdot(b\circ a)=b\circ a\circ \bar{b}.
\]
\end{lemma}

Since the socle $\Soc(B)$ is an ideal of a skew left brace, the quotient $(B,\circ)/\Soc(B)$ of the the group $(B,\circ)$ is also the quotient of the group $(B,\cdot)$ and the factor skew left brace $B/\Soc(B):=(B,\cdot,\circ)/\Soc(B)$ by $\Soc(B)$ is well defined.
\begin{de}
The socle series of $(B,\cdot,\circ)$ is defined as the sequence 
\[
B_0=B,\; B_{n+1}=B_n/\Soc(B_n), \; n\geq 0.
\]
We say that $(B,\cdot,\circ)$ is {\em nilpotent of class $n$} if $n$ is the least number such that $|B_n|=1$.
\end{de}

\begin{thm}\cite[Theorem 4.21]{SV18}\label{th:nil}
Let $(B,\cdot,\circ)$ be a nilpotent skew left brace. Then the group $(B,\cdot)$ is nilpotent. 
\end{thm}

In the rest of the section we shall focus on the
correspondence between the nilpotency of skew left braces and the multipermutation level of the associated
solutions. For involutive solutions, the correspondence was described by Rump in~\cite{Rump07A}.
In the involutive case, the mapping~$\rho$ is uniquely determined by~$\lambda$ hence there was no need to consider this mapping. In the non-involutive case, the relation~$\approx$ depends both on $\lambda$ and $\rho$ and hence we need to take a closer look at the role of~$\rho$ in the skew left brace.

\begin{proposition}\cite[Proposition 1.9]{GV}\label{prop:lambda_hom}
Let $(B,\cdot,\circ)$ be a skew left brace. The mapping 
$
\lambda\colon B\to B; \quad a\mapsto \lambda_a
$
is a homomorphism of the groups $(B,\circ)$ and $\aut{B,\cdot}$. 
\end{proposition}

The mapping $\rho_a$ is not an automorphism in general, as we can see on the following example.

\begin{example}
 Let $(B,+)=\Z_2^3$ and let us denote $e_0=(0,0,0)$, $e_1=(1,0,0)$, $e_2=(0,1,0)$, $e_3=(0,0,1)$ and $f=(1,1,1)$. Then each $b\in B$ can be uniquely written as $b=\varepsilon f+e_i$,
 for some $\varepsilon\in\Z_2$ and $i\in\Z_4$.
 Define now
 \[(\varepsilon f+e_i)\circ (\zeta f+e_j)=(\varepsilon +\zeta)f+e_{i+3^\varepsilon j} \qquad \text{and}\qquad \overline{\varepsilon f+e_i}=\varepsilon f+e_{3^{1+\varepsilon} i},
  \]
for any $\varepsilon,\zeta\in\Z_2$ and $i,j\in\Z_4$. It is easy to check that $(B,\circ)$ is the 8-element dihedral group and it is more complicated but still straightforward to check that
$(B,+,\circ)$ is a skew left brace (of abelian type). Then, for its associated solution, we have: 
$\rho_{f+e_1}(f+e_1)=f+e_1$, $\rho_{f+e_1}(f)=f$
and $\rho_{f+e_1}(e_1)=f+e_3$
but $(f+e_1)+f\neq f+e_3$ and therefore $\rho_{f+e_1}$ is not an endomorphism of~$(B,+)$.
Note that an equivalent but different presentation of this left brace was given by Bachiller in~\cite{B15}.
\end{example}

Nevertheless, 
even if the mappings $\rho_a$ are not endomorphisms, Bachiller showed that we can still have a partial
counterpart of~Proposition~\ref{prop:lambda_hom} for the mapping~$\rho$.

\begin{proposition}\cite[Lemma 2.4]{B18}\label{prop:tau_hom}
 Let $(B,\cdot,\circ)$ be a skew left brace. The mapping 
$
\rho\colon B\to B; \quad a\mapsto \rho_a
$
is an anti-homomorphism of the groups $(B,\circ)$ and $S_B$, that means $\rho_{a\circ b}=\rho_b\rho_a$, for all $a,b\in B$.
\end{proposition}

\begin{proposition}\cite[Proposition 2.8]{B18}\label{prop:SocKer}
 Let $(B,\cdot,\circ)$ be a skew left brace
 and let $(B,\lambda,\rho)$ its associated solution. Then
 $\Soc(B)=\Ker\lambda \cap \Ker\rho$.
 \end{proposition}

\begin{corollary}\cite[Remark 4.4]{CJKAV}\label{cor:RetSoc}
 Let $(B,\cdot,\circ)$ be a skew left brace
 and let $(B,\lambda,\rho)$ its associated solution. Then
 \begin{itemize}
  \item[(i)] $\mathrm{Ret}(B,\lambda,\rho)$ is the solution
  associated to $(B,\cdot,\circ)/\Soc(B)$,
  \item[(ii)] $(B,\lambda,\rho)$ is a multipermutation solution of level~$n$ if and only if $(B,\cdot,\circ)$ is nilpotent of class~$n$.
 \end{itemize}
\end{corollary}

\begin{de}\cite[Definition 2.11]{BNY}
A skew left brace $(B,\cdot,\circ)$ is called \emph{meta-trivial} if there exists an ideal $N$ of~$(B,\cdot,\circ)$ such that $(N,\cdot,\circ)$ is a trivial skew left brace and  the quotient skew left brace $B/N$ is trivial as well.
\end{de}

\begin{corollary}
Let $(B,\cdot,\circ)$ be a skew left brace. If its associated solution is $2$-reductive then $(B,\cdot,\circ)$ is meta-trivial.
\end{corollary}
\begin{proof}
It follows directly by Corollaries \ref{lm:sf} and  \ref{cor:RetSoc}(i) and the fact that the socle of $(B,\cdot,\circ)$ is a trivial skew left brace (of abelian type).
\end{proof}

\section{Distributive solutions}\label{sec:distr}

In~\cite{JPZ20} the authors and Zamojska-Dzienio were studying so called distributive solutions.
In this section we establish a connection
between bi-skew left braces and results
about distributive solutions.

\begin{de}
We will say that a solution $(X,\sigma,\tau)$ is \emph{left distributive}, if for every $x,y \in X$:
\begin{align}\label{eq:left}
\sigma_x\sigma_y=\sigma_{\sigma_x(y)}\sigma_x,
\end{align}
and it is \emph{right distributive}, if for every $x,y \in X$:
\begin{align}\label{eq:right}
\tau_x\tau_y=\tau_{\tau_x(y)}\tau_x.
\end{align}
A solution is \emph{distributive} if it is left and right distributive. 
\end{de}

\begin{example}[Lyubashenko, see \cite{Dr90}]\label{ex:Lyub}
Let $X$ be a non-empty set and let $f,g\colon X\to X$ be two bijections such that $fg=gf$. Then the permutational solution $(X,\sigma,\tau)$ with $\sigma_x=f$ and $\tau_y=g$, for each $x,y\in X$, is distributive.

\end{example}

\begin{example}
Let $(X,\sigma,\id)$ be a solution. Then by \eqref{birack:1} it is distributive. By \eqref{birack:3} the same we have for a solution $(X,\id,\tau)$. Hence, solutions from Examples \ref{exm:trisol} and \ref{exm:almtrisol} are both distributive.
\end{example}

\begin{example}
The solution from Example \ref{exm:Z2n} is left distributive but not right distributive.
\end{example}

\begin{theorem}\label{thm:dislbilred}
Let $(X,\sigma,\tau)$ be a solution. Then the following conditions are equivalent:
\begin{enumerate}
\item[(i)] $(X,\sigma,\tau)$ is left distributive,
\item[(ii)] $(X,\sigma,\tau)$ satisfies \eqref{eq:red3}, that means $\sigma_{\tau_x(y)}=\sigma_y$, for all $x,y\in X$,
\item[(iii)] $(X,\sigma,\tau)$ satisfies \eqref{eq:solbi}, that means $\sigma_{\hat\sigma_x(y)}=\sigma_y$, for all $x,y\in X$,
\item[(iv)] $(X,\sigma,\tau)$ satisfies $\hat{\tau}_x=\sigma^{-1}_x$, for all $x\in X$,
\item[(v)] $\LMlt(X)\leq \Aut(X,\sigma,\tau)$.
\end{enumerate}
\end{theorem}

\begin{proof}
By \cite[Lemma 2.8]{JPZ20}, the following conditions are equivalent:
\begin{align}\label{eq:disl}
&\sigma_x\sigma_y=\sigma_{\sigma_x(y)}\sigma_x\quad \Leftrightarrow\quad \sigma_x=\sigma_{\tau_y(x)}\quad \Leftrightarrow\quad \sigma_x=\sigma_{\tau^{-1}_y(x)}.
\end{align}
This is actually (i)$\Leftrightarrow$(ii). The equivalence (ii)$\Leftrightarrow$(iii)
was proved in Proposition~\ref{prop:blr}. Now
\begin{align}
\hat{\tau}^{-1}_x(y)\stackrel{\eqref{rr:4}}=\sigma_{\tau^{-1}_y(x)}(y)\stackrel{\eqref{eq:red3}}=\sigma_x(y)\quad \Leftrightarrow\quad \hat{\tau}_x=\sigma^{-1}_x,
\end{align}
proving (ii)$\Leftrightarrow$(iv). Finally, for each $x\in X$, the mapping $\sigma_x$ is a permutation and (v)$\Rightarrow$(i)
follows from the definition. On the other hand,
\[\tau_{\sigma_{\tau_y(x)}(z)}\sigma_x(y)\stackrel{\eqref{birack:2}}=\sigma_{\tau_{\sigma_y(z)}(x)}\tau_{z}(y)
\quad
\stackrel{\eqref{eq:red3}}\Rightarrow\quad \tau_{\sigma_{x}(z)}\sigma_x=\sigma_{x}\tau_{z}\]
proving (i)\&(ii)$\Rightarrow$(v).
\end{proof}

\begin{corollary}\label{corol:left_dist}
Let $(B,\cdot,\circ)$ be a skew left brace. Then $(B,\cdot,\circ)$ is a bi-skew left brace if and only if its associated solution $(B,\lambda,\rho)$ satisfies
one of the properties from Theorem~\ref{thm:dislbilred}.
\end{corollary}

Analogously we can prove the mirror image of Theorem \ref{thm:dislbilred}.

\begin{theorem}\label{thm:disrbilred}
Let $(X,\sigma,\tau)$ be a solution. Then the following conditions are equivalent:
\begin{enumerate}
\item[(i)] $(X,\sigma,\tau)$ is right distributive,
\item[(ii)] $(X,\sigma,\tau)$ satisfies \eqref{eq:red4}, that means $\tau_{\sigma_x(y)}=\tau_y$, for all $x,y\in X$,
\item[(iii)] $(X,\sigma,\tau)$ satisfies \eqref{eq:solbir}, that means $\tau_{\hat\tau_x(y)}=\tau_y$, for all $x,y\in X$,
\item[(iv)] $(X,\sigma,\tau)$ satisfies $\hat{\sigma}_x=\tau^{-1}_x$, for all $x\in X$,
\item[(v)] $\RMlt(X)\leq \Aut(X,\sigma,\tau)$.
\end{enumerate}
\end{theorem}
 
\begin{corollary}\label{corol:left_right}
A solution $(X,\sigma,\tau)$ is left distributive if and only if $(X,\hat{\sigma},\hat{\tau})$ is right distributive.
\end{corollary}
\begin{proof}
Let $(X,\sigma,\tau)$ be a left distributive solution. By Theorem~\ref{thm:dislbilred}, we have $\hat{\tau}_y=\sigma^{-1}_y$, for $y\in X$ . Then
\begin{align*}
\hat{\tau}_{\tau_y(x)}=\sigma^{-1}_{\tau_y(x)}\stackrel{\eqref{eq:red3}}=\sigma^{-1}_x=\hat{\tau}_x,
\end{align*}
which is equivalent to right distributivity of $(X,\hat{\sigma},\hat{\tau})$.

If a solution $(X,\hat{\sigma},\hat{\tau})$ is right distributive we have $\sigma_y=\hat{\tau}^{-1}_y$ which implies
\begin{align*}
\sigma_{\hat{\sigma}_y(x)}=\hat{\tau}^{-1}_{\hat{\sigma}_y(x)}\stackrel{\eqref{eq:red4}}=\hat{\tau}^{-1}_x=\sigma_x
\end{align*}
and shows that $(X,\sigma,\tau)$ is left distributive.
\end{proof}

By Theorem~\ref{thm:Koch}, Corollary~\ref{corol:left_dist} and Corollary~\ref{corol:left_right} we obtain the following:
\begin{corollary}\label{cor:opbi}
 Let $(B,\cdot,\circ)$ be a skew left brace and $(B,\lambda,\rho)$ its associated solution. Then $(B,\cdot_{op},\circ)$ is a bi-skew left brace if and only if $(B,\lambda,\rho)$ satisfies \eqref{eq:solbir}, that means $\rho_{\hat\rho_a(b)}=\rho_b$, for all $a,b\in B$.
\end{corollary}

\begin{corollary}
 Let $(B,\cdot,\circ)$ be a skew left brace and $(B,\lambda,\rho)$ its associated solution. Then $(B,\lambda,\rho)$ is distributive if and only if $(B,\circ,\cdot)$ is a skew two-sided brace.
\end{corollary}

\begin{proof}
Recall, a solution  $(B,\lambda,\rho)$ is left distributive if and only if $(B,\cdot,\circ)$ is a bi-skew left brace which is equivalent to $(B,\circ,\cdot)$ being a skew left brace.
 
A solution $(B,\lambda,\rho)$ is right distributive if and only if 
 $(B,\hat\lambda,\hat\rho)$ is left distributive if and only if $(B,\cdot_{op},\circ)$ is a bi-skew left brace, according
 to Corollary \ref{cor:opbi} and Theorem \ref{thm:disrbilred}.
 This is equivalent to satisfying Condition \eqref{eq:twosided}.
\end{proof}

\begin{lemma}
Let $(X,\sigma,\tau)$ be a left distributive solution. Then for $x,y\in X$ 
\begin{align}
&\tau_y\tau_x=\tau_{\hat{\tau}_y(x)}\tau_{\hat{\sigma}_x(y)}.\label{eq:hathattau}
\end{align}
\end{lemma}

\begin{proof}
For $x,y\in X$ we have
\begin{align*}
&\tau_y\stackrel{\eqref{rr:2}}=\tau_{\tau_{\hat{\tau}_y(x)}\hat{\sigma}_x(y)}\stackrel{\eqref{birack:3}}=\tau_{\hat{\tau}_y(x)}\tau_{\hat{\sigma}_x(y)}\tau^{-1}_{\sigma_{\hat{\sigma}_x(y)}\hat{\tau}_y(x)}
\stackrel{\text{Th.~\ref{thm:dislbilred}}}=
\tau_{\hat{\tau}_y(x)}\tau_{\hat{\sigma}_x(y)}\tau^{-1}_{\sigma_y\sigma^{-1}_y(x)}=
\tau_{\hat{\tau}_y(x)}\tau_{\hat{\sigma}_x(y)}\tau^{-1}_x.\qedhere
\end{align*}
\end{proof}

We have discussed in Section~\ref{sec:prelim} that the equivalences $\sim$ and $\backsim$ are not congruences in general. Nevertheless, they have to be congruences of distributive solutions:

\begin{theorem}\cite[Theorem 3.4]{JPZ20a}
 Let $(X,\sigma,\tau)$ be a solution.
 If $(X,\sigma,\tau)$ is left distributive then $\sim$ is its congruence. If $(X,\sigma,\tau)$ is right distributive then $\backsim$ is its congruence.
\end{theorem}

For bi-skew left braces, we have a
left-sided
analogy to Corollary~\ref{cor:RetSoc}.

\begin{proposition}\label{prop:Ker_lambda}
 Let $(B,\cdot,\circ)$ be a bi-skew left brace and $(B,\lambda,\rho)$ its associated solution. Then 
 \begin{itemize}
  \item[(i)] $\Ker\lambda$ is an ideal of~$(B,\cdot,\circ)$, 
   \item[(ii)] $\mathrm{LRet}(B,\lambda,\rho)$ is the solution associated to $(B,\cdot,\circ)/\Ker\lambda$.
 \end{itemize}
\end{proposition}

\begin{proof}
 Let $a,b\in \Ker\lambda$ and $x,z\in B$. Then $\lambda_a(x)=x$ which means $a\cdot x=a\circ x$ and $a^{-1}=\bar a$.
 
(i) At first we will show that $\Ker\lambda$ is an ideal of a skew left brace $(B,\cdot,\circ)$.   \begin{itemize}
 \item $\Ker\lambda$ is a normal subgroup of~$(B,\circ)$ since
 $\lambda$ is a homomorphism $(B,\circ)\to\Aut(B,\cdot)$;
  \item $\lambda_{a\cdot b}=\lambda_{a\circ b}=\lambda_a\lambda_b=\mathrm{id}$ which proves that $\Ker\lambda$ is a subgroup of $(B,\cdot)$;
  \item $\lambda_{x^{-1}\cdot a\cdot x}(z)=
  (x^{-1}\cdot a\cdot x)^{-1}\cdot((x^{-1}\cdot a\cdot x)\circ z)
  =x^{-1}\cdot a^{-1}\cdot x\cdot((x^{-1}\cdot a\cdot x)\circ z)=\strut$\newline
  $x^{-1}\cdot \bar a\cdot ((a\cdot x)\circ \bar x\circ (x\cdot z))=x^{-1}\cdot (\bar a\circ a\circ x\circ \bar x\circ (x\cdot z))=x^{-1}\cdot x\cdot z=z$
  which proves that $\Ker\lambda$ is a normal subgroup of~$(B,\cdot)$;  
  \item $\lambda_{\lambda_x(a)}\stackrel{\eqref{eq:left}}{=}\lambda_x\lambda_a\lambda_x^{-1}=\mathrm{id}$ and therefore $\lambda_x(\Ker\lambda)\subseteq \Ker\lambda$, for all $x\in X$.
 \end{itemize}

(ii) Now we want to prove that the cosets of $\Ker\lambda$ coincide with the classes of~$\sim$. But
$$x\circ \bar y\in\Ker\lambda\quad \Leftrightarrow\quad
\lambda_{x\circ\bar y}=\mathrm{id}\quad \Leftrightarrow\quad \lambda_x=\lambda_y
\quad \Leftrightarrow\quad x\sim y. \qquad \qedhere$$
\end{proof}

The mirror version of Proposition~\ref{prop:Ker_lambda} is not
that straightforward since $\rho_a$ 
may behave differently than $\lambda_a$,
for instance, it is
not an automorphism for general skew left braces.

\begin{lemma}\label{lm:tau_right}
 Let $(B,\cdot,\circ)$ be a skew left brace 
 and let $(B,\lambda,\rho)$ be its associated solution. 
 If the solution $(B,\lambda,\rho)$ is right distributive then 
 \begin{itemize}
  \item[(i)] $\rho_y(x)=\bar{y}\circ(x\cdot y)$, for all $x,y\in B$,
  \item[(ii)]  $\rho_x$ is an anti-automorphism of $(B,\cdot)$, for all~$x\in B$.  
 \end{itemize}
\end{lemma}
\begin{proof}
Let $x,y\in B$. According to~Theorem~\ref{thm:disrbilred}, $\rho^{-1}_y=\hat\lambda_y$.

(i) By Theorem  \ref{thm:Koch}  we have
\[
\rho^{-1}_y(x)=\hat{\lambda}_y(x)=(y\circ x)\cdot y^{-1}.
\]
It is easy to verify that
\[
\rho_y^{-1}(\bar{y}\circ (x\cdot y))=
(y\circ\bar{y}\circ (x\cdot y))\cdot y^{-1}
=x\quad {\rm and}\quad \rho_y((y\circ x)\cdot y^{-1})=\bar{y}\circ ((y\circ x)\cdot y^{-1}\cdot y)=x,
\]
which completes the proof.
 
 (ii) 
 The mapping~$\hat\lambda$ is a homomorphism $(B,\circ)\to\Aut(B,\cdot_{op})$, according to Theorem~\ref{thm:Koch} and Proposition \ref{prop:lambda_hom}. Therefore, for all $a,b,x\in B$,
 \[ \rho_x^{-1}(a\cdot b)=\hat\lambda_x(b\cdot_{op} a)
  =\hat\lambda_x(b)\cdot_{op}\hat\lambda_x(a)=\rho_x^{-1}(a)\cdot\rho_x^{-1}(b).
 \]
 and $\rho_x$ is an automorphism of~$(B,\cdot)$.
 
\end{proof}

\begin{theorem}\label{thm:ker_tau}
 Let $(B,\cdot,\circ)$ be a skew left brace, 
$(B,\lambda,\rho)$ be its associated solution and $(B,\cdot_{op},\circ)$ be a bi-skew left brace. 
 Then 
 \begin{itemize}
  \item[(i)] $\Ker\rho$ is an ideal of $(B,\cdot,\circ)$,
  \item[(ii)] $\mathrm{RRet}(B,\cdot,\circ)$ is the solution
  associated to $(B,\cdot,\circ)/\Ker\rho$.   
 \end{itemize}
\end{theorem}

\begin{proof}
Let $x,y\in B$. By Corollary \ref{cor:opbi} and Theorem \ref{thm:disrbilred} the solution $(B,\lambda,\rho)$ is right distributive and $\rho_x=\hat\lambda_x^{-1}$ for all $x\in B$. Hence

 (i) $\Ker\rho=\Ker\hat\lambda^{-1}$ and this is a normal subgroup of both $(B,\cdot)$ as well as $(B,\circ)$, according to Proposition~\ref{prop:Ker_lambda}. Now
 $\rho_{\lambda_x(y)}=\rho_y$, according to Theorem~\ref{thm:disrbilred} (ii) and this means that
 $\lambda_x(\Ker\rho)\subseteq \Ker\rho$, for all $x\in B$.

 (ii) $x\circ \bar y\in\Ker\rho\quad \Leftrightarrow\quad
\rho_{x\circ\bar y}=\mathrm{id}\quad \Leftrightarrow\quad \rho_x=\rho_y
\quad \Leftrightarrow\quad x\backsim y$.
\end{proof}

\begin{example}
Consider the bi-skew left brace $(\Z_{2n},\cdot,+_{2n})$
defined in Example \ref{exm:Z2n}. Let $(\Z_{2n},\lambda,\rho)$ be its associated left distributive solution.
The mapping~$\rho$ is an anti-homomorphism of the groups $(\Z_{2n},\cdot)$ and $S_{\Z_{2n}}$,
according to Proposition~\ref{prop:tau_hom} but
$\Ker\rho=\{0,n\}$ is not a normal subgroup since
$1^{-1}\cdot n\cdot 1=2-n\notin \{0,n\}$.\newline

The equivalence~$\sim$ is a congruence, namely,
for $x,y\in \Z_{2n}$,
\begin{align*}
x\sim y\quad \Leftrightarrow\quad \forall(a\in \Z_{2n})\; \lambda_x(a)=\lambda_y(a)\quad \Leftrightarrow\\
& \forall(a\in \Z_{2n})\; (-1)^{x}a=(-1)^ya\quad \Leftrightarrow\quad x\equiv y\pmod 2.
\end{align*}
The quotient solution $(\Z_{2n}^{\sim},\lambda,\rho)$ is a two element trivial solution.
On the other hand
\begin{multline*}
x\backsim y\quad \Leftrightarrow\quad \forall(a\in \Z_{2n})\; \rho_x(a)=\rho_y(a)\quad \Leftrightarrow\\
 \forall(a\in \Z_{2n})\; (-1)^{a+1}x+x+a=(-1)^{a+1}y+y+a\quad \Leftrightarrow\quad \\
\forall(a\in \Z_{2n})\; x(1+(-1)^{a+1})=y(1+(-1)^{a+1})\quad \Leftrightarrow\quad
2x=2y \quad \Leftrightarrow\quad
x\equiv y\pmod n.
\end{multline*}
But $\lambda_0(1)=1$ and  $\lambda_n(1)=-1$ are not $\backsim$ related, hence the induced solution is not defined on the quotient $\Z_{2n}^{\backsim}$.
Finally,
\begin{align*}
&x\mathrel{\approx} y\quad \Leftrightarrow\quad
x\sim y \; \wedge\; x\backsim y \quad \Leftrightarrow\quad x\equiv y\pmod 2 \; \wedge\; x\equiv y\pmod n
\quad \Leftrightarrow\quad x=y,
\end{align*} 
which means that the solution $(\Z_{2n},\lambda,\rho)$ is irretractable. It corresponds to the fact $\Soc(\Z_{2n})=\{0\}$.
\end{example}

\begin{example}
Let $(G,\cdot)$ be a non-abelian group and let $(G\times G,\lambda,\rho)$ be the solution defined in Example \ref{exm:prodsol}. 
This solution is clearly distributive and therefore all the three equivalences $\sim$, $\backsim$ and $\approx$ are congruences.
In particular, for $(x,y),(a,b)\in G\times G$ 
\begin{multline*}
(x,y)\sim (a,b)\quad \Leftrightarrow\quad \forall((u,w)\in G\times G)\; \lambda_{(x,y)}((u,w))=\lambda_{(a,b)}((u,w))\quad \Leftrightarrow\\
 \forall(w\in G)\; y^{-1}\cdot w \cdot y=b^{-1} \cdot w \cdot  b\quad \Leftrightarrow\quad y \cdot b^{-1}\in Z(G).
\end{multline*}
Analogously
$(x,y)\backsim (a,b)\ \Leftrightarrow\  x \cdot a^{-1}\in Z(G)$.
Combining, we obtain
\begin{align*}
&(x,y)\mathrel{\approx} (a,b)\quad \Leftrightarrow\quad
(x,y)\sim (a,b) \; \wedge\; (x,y)\backsim (a,b) \quad \Leftrightarrow\quad x \cdot a^{-1}, y \cdot b^{-1}\in Z(G).
\end{align*} 
In the case of $Z(G)$ being trivial, the solution $(G\times G,\lambda,\rho)$ is irretractable.
\end{example}

\section{2-reductive skew left braces}\label{sec:last}

In this section we analyze which skew left braces
yield $2$-reductive solutions. Since these solutions
are distributive, we can specialize all the results of the previous section.
The first result tells how the identities of $2$-reductivity translate to properties of skew left braces.

\begin{proposition}\label{prop:4ekv2red}
Let $(B,\cdot,\circ)$ be a skew left brace and $(B,\lambda,\rho)$ its associated solution. Then
\begin{enumerate}
\item [(i)] $(B,\lambda,\rho)$ satisfies \eqref{eq:red1}, that means $\lambda_{\lambda_x(y)}=\lambda_y$, for all $x,y\in B$, if and only if the mapping $\lambda\colon B\to  B; \; a\mapsto \lambda_a
$
is a homomorphism of groups $(B,\cdot)$ and $\aut{B,\cdot}$,
\item [(ii)] $(B,\lambda,\rho)$ satisfies \eqref{eq:red2}, that means $\rho_{\rho_x(y)}=\rho_y$, for all $x,y\in B$, if and only if the mapping $\rho\colon B\to B; \; a\mapsto \rho_a
$
is a homomorphism of groups $(B,\cdot)$ and $S_B$,
\item [(iii)] $(B,\lambda,\rho)$ satisfies \eqref{eq:red3}, that means $\lambda_{\rho_x(y)}=\lambda_y$, for all $x,y\in B$, if and only if the mapping $\lambda\colon B\to B; \; a\mapsto \lambda_a
$
is an anti-homomorphism of groups $(B,\cdot)$ and $\aut{B,\cdot}$,
\item [(iv)] $(B,\lambda,\rho)$ satisfies \eqref{eq:red4}, that means $\rho_{\lambda_x(y)}=\rho_y$, for all $x,y\in B$, if and only if the mapping $\rho\colon B\to B; \; a\mapsto \rho_a
$
is an anti-homomorphism of groups $(B,\cdot)$ and $\aut{B,\cdot}$,
\end{enumerate}
\end{proposition}
\begin{proof}
By Remark \ref{prop:lambda-hom} and Proposition \ref{prop:tau_hom} we have for $x,y\in B$:
\begin{enumerate}
\item [(i)] $\lambda_{\lambda_x(y)}=\lambda_y\quad \Leftrightarrow\quad \lambda_y=\lambda_{\lambda^{-1}_x(y)}=\lambda_{\bar{x}\circ (x\cdot y)}=\lambda^{-1}_x\lambda_{x\cdot y}\quad \Leftrightarrow\quad \lambda_{x\cdot y}=\lambda_x\lambda_y$;
\item [(ii)] $\rho_{x}=\rho_{\rho_y(x)} 
=\rho_{\overline{\lambda_{x}(y)}\circ x\circ y}=\rho_y\rho_{x}\rho_{\overline{\lambda_{x}(y)}}=\rho_y\rho_{x}\rho^{-1}_{\lambda_{x}(y)}\quad \Leftrightarrow\quad
\rho_y\rho_{x}=\rho_{x}\rho_{\lambda_{x}(y)}
\quad\Leftrightarrow$
\newline 
$\rho_x\rho_y=\rho_{\lambda^{-1}_x(y)}\rho_x
=\rho_{\bar x\circ(x\cdot y)}\rho_x
=\rho_{x\cdot y}\rho^{-1}_x\rho_x
=\rho_{x\cdot y}$;

\item [(iii)] Follows from Lemma \ref{lem:Car} and Theorem \ref{thm:dislbilred} but we give a direct proof anyway since it is analogous to the proof of (ii):
$\lambda_{x}=\lambda_{\tau_y(x)} 
=\lambda_{\overline{\lambda_{x}(y)}\circ x\circ y}
=\lambda^{-1}_{\lambda_{x}(y)}\lambda_x\lambda_y
\quad \Leftrightarrow$\newline
$\lambda_{x}\lambda_y=\lambda_{\lambda_{x}(y)}\lambda_x
\quad\Leftrightarrow\quad
\lambda_y\lambda_x=\lambda_x\lambda_{\lambda^{-1}_x(y)}
=\lambda_x\lambda_{\bar x\circ(x\cdot y)}
=\lambda_{x\cdot y}$;
\item [(iv)] $\rho_{\lambda_x(y)}=\rho_y\quad \Leftrightarrow\quad \rho_y=\rho_{\lambda^{-1}_x(y)}=\rho_{\bar{x}\circ (x\cdot y)}=\rho_{x\cdot y}\rho^{-1}_x\quad \Leftrightarrow\quad \rho_y\rho_x=\rho_{x\cdot y}$
and $\rho_x$ is an automorphism according to Theorem~\ref{thm:ker_tau}.\qedhere
\end{enumerate}
\end{proof}

Skew left braces $(B,\cdot,\circ)$ for which the mapping $\lambda$ is a homomorphism of groups $(B,\cdot)$ and $\aut{B,\cdot}$ were investigated in \cite{BNY} under the name $\lambda$-\emph{homomorphic} skew left braces. In particular, it was shown there (\cite[Theorem 2.12]{BNY}) that any $\lambda$-homomorphic skew left brace is meta-trivial. Hence, by Proposition \ref{prop:4ekv2red}(i) we immediately obtain:
\begin{corollary}
Each skew left brace $(B,\cdot,\circ)$ which satisfies \eqref{eq:red1}, that means $\lambda_{\lambda_x(y)}=\lambda_y$, for all $x,y\in B$, is meta-trivial.
\end{corollary}

\begin{lemma}
Let $(B,\lambda,\rho)$ be a $2$-reductive solution associated to a skew left brace $(B,\cdot,\circ)$. Then for $x,y\in B$ we have 

\begin{enumerate}
\item [(i)] $y\circ y=y\cdot\overline{y^{-1}}$,
\item [(ii)] $(y^{-1}\circ x)\cdot y=(\bar{y}\circ x)\cdot \bar{y}^{-1}=\rho_y(x)$,
\item [(iii)] $\bar{y}\cdot y=y\cdot \bar{y}$,
\item [(iv)] $y^{-1}\cdot \bar{y}^{-1}=y^{-1}\circ y$,
\item [(iv)] $y\circ y=\overline{y^{-1}}\circ \bar{y}^{-1}$,
\item [(vi)] $\overline{\overline{y}\cdot y}=(\bar{y}\cdot y)^{-1}=y^{-1}\circ y$,
\item [(vii)] $x\circ y\circ \bar{x}\circ\bar{y}=(x\circ y)\cdot(y\circ x)^{-1}$.
\end{enumerate}
\end{lemma}

\begin{proof}
Let $x,y\in B$. \\
(i) By Lemma \ref{lem:Car} we have the following:
\[
y^{-1}\cdot(y\circ y)
 =\lambda_y(y)= 
\lambda_{y^{-1}}^{-1}(y)=\overline{y^{-1}}\circ (y^{-1}\cdot y)=\overline{y^{-1}}\quad \Rightarrow\quad y\circ y=y\cdot\overline{y^{-1}}.
 \]
(ii) By Lemma \ref{lm:tau_right} we obtain
\begin{align}\label{eq:tau}
\rho_y(x)=\bar{y}\circ(x\cdot y)\stackrel{\eqref{lsb}}=(\bar{y}\circ x)\cdot \bar{y}^{-1}.
\end{align}
Hence, by $2$-reductivity
\[
\rho_y(x)\stackrel{\eqref{eq:red4}}{=}\rho_{\lambda_{y^{-1}}^{-1}(y)}(x)=\rho_{\overline{y^{-1}}}(x)=(y^{-1}\circ x)\cdot y.
\]
(iii) Once again by Lemma \ref{lm:tau_right}: 
\[
\overline{y^{-1}}=\overline{y^{-1}}\circ (y\cdot y^{-1})=\rho_{y^{-1}}(y)\stackrel{(\text {ii})}{=}(y\circ y)\cdot y^{-1}
\quad \Rightarrow\quad \overline{y^{-1}}\cdot y=y\circ y\stackrel{(\text{i})}{=}y\cdot \overline{y^{-1}}\quad \Rightarrow\quad \bar{y}\cdot y=y\cdot \bar{y}.
\]
(iv) By Lemma \ref{lem:Car} 
\begin{align*}
&y^{-1}=\lambda_y(\bar{y})
=\lambda_{\bar{y}^{-1}}(\bar{y})=\bar{y}\cdot (\bar{y}^{-1}\circ \bar{y})\quad \Rightarrow\quad \bar{y}^{-1} \cdot y^{-1} =\bar{y}^{-1}\circ \bar{y}\quad \stackrel{y\mapsto\bar{y}}{\Rightarrow}\quad 
y^{-1} \cdot \bar{y}^{-1} =y^{-1}\circ y.
\end{align*}
(v) Conditions (iii) and (iv) imply:
\begin{align*}
&y^{-1}\circ y\stackrel{(\text {iv})}{=}y^{-1}\cdot \bar{y}^{-1}=(\bar{y}\cdot y)^{-1}
\stackrel{(\text {iii})}{=}(y\cdot \bar{y})^{-1}
=\bar{y}^{-1}\cdot y^{-1}\stackrel{(\text {iv})}{=}\bar{y}^{-1}\circ \bar{y}\quad \Rightarrow\\
&
y^{-1}\circ y\circ y=\bar{y}^{-1}\quad \Rightarrow\quad  \overline{y^{-1}}\circ \bar{y}^{-1}=y\circ y.
\end{align*}
(vi) By Lemma \ref{lem:Car} we have:
\begin{align*}
&y\circ(\bar{y}\cdot y)=\lambda_{\bar{y}}^{-1}(y)=\lambda_y(y)=\overline{y^{-1}}\quad \Rightarrow\\
& \bar{y}\cdot y=\bar{y}\circ \overline{y^{-1}}=\overline{y^{-1}\circ y}\quad \stackrel{(\text {iv})}{=}\overline{(\bar{y}\cdot y)^{-1}}\quad \Rightarrow\quad \overline{\overline{y}\cdot y}=(\bar{y}\cdot y)^{-1}\stackrel{(\text {iv})}   {=}y^{-1}\circ y.
\end{align*}
(vii) By $2$-reductivity the group $\LMlt(X)$ is commutative. Then
\begin{align*}
&x^{-1}\cdot(x\circ(y^{-1}(y\circ \bar{x})))=\lambda_x(y^{-1}\cdot(y\circ \bar{x}))= \lambda_x\lambda_y(\bar{x})=\\
&\lambda_y\lambda_x(\bar{x})=\lambda_y(x^{-1}\cdot(x\circ \bar{x}))=\lambda_y(x^{-1})=y^{-1}\cdot(y\circ x^{-1}).
\end{align*}
By Lemma \ref{lm:GV1.7} we have:
\begin{align*}
&x^{-1}\cdot(x\circ(y^{-1}\cdot(y\circ \bar{x})))\stackrel{\eqref{eq:1.7(2)}}{=}x^{-1}\cdot x\cdot(x\circ y)^{-1}\cdot(x\circ y\circ \bar{x})=(x\circ y)^{-1}\cdot(x\circ y\circ \bar{x})
\end{align*}
and
\begin{align*}
&y^{-1}\cdot(y\circ x^{-1})=y^{-1}\cdot(y\circ (1\cdot x^{-1}))\stackrel{\eqref{eq:1.7(3)}}{=}y^{-1}\cdot(y\circ 1)\cdot(y\circ x)^{-1}\cdot y=\\
&y^{-1}\cdot y\cdot(y\circ x)^{-1}\cdot y=(y\circ x)^{-1}\cdot y.
\end{align*}
Hence
\begin{align}
&(x\circ y)^{-1}\cdot(x\circ y\circ \bar{x})=(y\circ x)^{-1}\cdot y\quad \Rightarrow\nonumber\\
& 
x\circ y\circ \bar{x}=(x\circ y)\cdot (y\circ x)^{-1}\cdot y.\label{eq:com}
\end{align}
Substituting $y$ by $\bar{x}$ and $x$ by $y$ in \eqref{eq:com} we complete the proof:
\begin{align*}
&y\circ \bar{x}\circ \bar{y}=(y\circ\bar{x})\cdot(\bar{x}\circ y)^{-1}\cdot\bar{x}\quad \Rightarrow\quad x\circ y\circ \bar{x}\circ \bar{y}=x\circ ((y\circ\bar{x})\cdot(\bar{x}\circ y)^{-1}\cdot\bar{x})\stackrel{\eqref{lsb}}{=}\\
&(x\circ ((y\circ \bar{x})\cdot(\bar{x}\circ y)^{-1}))\cdot x^{-1}\cdot(x\circ\bar{x})\stackrel{\eqref{eq:1.7(3)}}{=}(x\circ y\circ \bar{x})\cdot(x\circ \bar{x}\circ y)^{-1}\cdot x\cdot x^{-1}=\\
&(x\circ y\circ \bar{x})\cdot y^{-1}\stackrel{\eqref{eq:com}}{=}(x\circ y)\cdot (y\circ x)^{-1}\cdot y\cdot y^{-1}=(x\circ y)\cdot (y\circ x)^{-1}. &&\qedhere
\end{align*}
\end{proof}

\begin{theorem}\label{thm:2-red_brace}
 Let $(B,\cdot,\circ)$ be a skew left brace and let $(B,\lambda,\rho)$ be its associated solution. Then 
 the following conditions are equivalent:
 \begin{itemize}
  \item[(i)] $(B,\lambda,\rho)$ is $2$-reductive,  
  \item[(ii)]  $\lambda_{a\cdot b}=\lambda_{b\cdot a}=\lambda_{a\circ b}$ and $\rho_{a\cdot b}=\rho_{b\cdot a}=\rho_{a\circ b}$, for all $a,b\in B$,
  \item[(iii)] $(B,\lambda,\rho)$ is of multipermutation level at most~$2$,

  \item[(iv)] $(B,\cdot,\circ)$ is nilpotent of class at most~$2$,
  \item[(v)] $(B,\cdot_{op},\circ)$ is nilpotent of class at most~$2$.
 \end{itemize}
\end{theorem}

\begin{proof}
(i)$\Leftrightarrow$(ii) follows from Proposition~\ref{prop:4ekv2red}, 
(iii)$\Leftrightarrow$(iv) follows from Corollary~\ref{cor:RetSoc}, 
(iii)$\Leftrightarrow$(i): We use Proposition~\ref{prop:star} since $\lambda_1(x)=x=\rho_1(x)$.

(iv)$\Leftrightarrow$(v): 
Condition (v) is now equivalent to $(B,\hat\lambda,\hat\rho)$ being $2$-reductive, according to Theorem~\ref{thm:Koch}.
But this is equivalent to $(B,\lambda,\rho)$ being $2$-reductive,
according to Proposition~\ref{prop:inv_2red}.
\end{proof}

\begin{corollary}
 Let $(B,\lambda,\rho)$ be a $2$-reductive solution associated to a skew left brace $(B,\cdot,\circ)$.
 Then $(B,\cdot)$ is nilpotent of class at most~$2$.
\end{corollary}

\begin{proof}
 We have $\Soc(B)\leq Z(B,\cdot)$. And, according to Theorem~\ref{thm:2-red_brace}, $(B,\cdot)/\Soc(B)$ is  abelian.
\end{proof}

In Section~\ref{sec:2red} we were constructing $2$-reductive solutions as the disjoint unions of abelian groups. There is hence a natural question whether these
abelian groups have some structural meaning in the skew left brace. The following example shows that it is not really so.

\begin{example}(See also \cite[Example 5.6]{BNY})
Let $(\Z_{2n},+_{2n},\cdot)$ be the skew left brace dual to the Example~\ref{exm:Z2n}, that means $a\cdot b=a+(-1)^ab$, for $a,b\in \Z_{2n}$. Then $\lambda_a(b)=\rho_a(b)=(-1)^ab$.
Since $[\Z_{2n}:\Soc(\Z_{2n})]=2$, the associated solution is 
$2$-reductive.
It is isomorphic to the union of
\[(\Z_1,\underbrace{\Z_2,\ldots,\Z_2}_{\frac{n-1}2\times},\Z_1,\underbrace{\Z_2,\ldots,\Z_2}_{\frac{n-1}2\times},\quad \scriptsize
\left(\begin{array}{cccc}
 0 & 0 & \cdots & 0\\
 &&\vdots\\
 0 & 0 & \cdots & 0\\
 1 & 1 & \cdots & 1\\
 &&\vdots\\
 1 & 1 & \cdots & 1
\end{array}\right)
,\quad
\left(\begin{array}{cccc}
 0 & 0 & \cdots & 0\\
 &&\vdots\\
 0 & 0 & \cdots & 0\\
 1 & 1 & \cdots & 1\\
 &&\vdots\\
 1 & 1 & \cdots & 1
\end{array}\right)).
\]
\end{example}

\vskip 3mm
\noindent
Recall that, for a solution $(X,\sigma, \tau)$, we can associate two groups:
\[G(X,r)=\langle X \mid x\circ y=\sigma_x(y)\circ \tau_y(x)\rangle
 \qquad\text{ and }\qquad
 A(X,r)=\langle X \mid x\cdot y=y\cdot\sigma_y\hat\sigma_y(x)\rangle.
\]
These two groups, when put across each other, form a skew left brace. 
A major part of the proof of the following theorem is almost a copy-paste of the proof of~\cite[Theorem 5.4]{ST}. It also directly follows from \cite[Theorem 3.13]{CT} just published by  Castelli and Trappeniers.

\begin{theorem}
 Let $(X,\sigma,\tau)$ be a $2$-reductive solution. Then $G(X,r)$ is a nilpotent skew left brace of class~$2$.
\end{theorem}

\begin{proof}
Let $(G(X,r),\lambda,\rho)$ be the solution  associated with the skew left brace $G(X,r)$. According to Theorem~\ref{thm:2-red_brace}, we need to prove $\lambda_{a\cdot b}=\lambda_{b\cdot a}=\lambda_a\lambda_b$
 and $\rho_{a\cdot b}=\rho_{b\cdot a}=\rho_a\rho_b$, for all~$a,b\in G(X,r)$. The part $\lambda_{b\cdot a}=\lambda_a\lambda_b$ follows from~\cite[Corollary 5.5]{ST}.
 
 Let us focus on $\lambda$ being a homomorphism.
 From~\eqref{eq:red1} we have 
 $\sigma_{\sigma_x(y)}=\sigma_y$, for all $x,y\in X$, and therefore
 $\lambda_{\lambda_x(y)}=\lambda_y$, for all $x,y\in X$. Since $X$ generates $G(X,r)$, we inductively obtain $\lambda_{\lambda_a(y)}=\lambda_y$, for all $a\in G(X,r)$ and $y\in X$.
 
 Let now $w=x_1^{e_1}x_2^{e_2}\cdot \cdots\cdot x_k^{e_k}$, where $x_i\in X$ and $e_i=\pm 1$, for $1\leq i\leq k$. We shall prove, by an induction on~$k$, that
 \[\lambda_w=\lambda_{x_1}^{e_1}\lambda_{x_2}^{e_2}\cdots \lambda_{x_k}^{e_k}.\]
 For $k=1$ and $e_1=1$ the claim is trivial. For $e_1=-1$ we remark
 \[\lambda_{a}(\bar a)=a^{-1} \quad\Rightarrow\quad
  a=\lambda_{a}(\bar a^{-1})
  \quad\Rightarrow\quad
  \bar a=\lambda_{\bar a}(a^{-1})
  \quad\Rightarrow\quad
  a=\overline{\lambda_{\bar a}(a^{-1})},
 \]
 hence
 \[
  \lambda_{x^{-1}}=\lambda_{\overline{\lambda_{\overline{x^{-1}}}(x)}}
  =\lambda_{\lambda_{\overline{x^{-1}}}(x)}^{-1}
  =\lambda_{x}^{-1}.
 \]
 Now we assume that the induction hypothesis is valid for $k-1$ and we compute, for
 $w=x_1^{e_1}x_2^{e_2}\cdot \cdots\cdot x_k^{e_k}$ and $v=x_1^{e_1}x_2^{e_2}\cdot \cdots\cdot x_{k-1}^{e_{k-1}}$
 \[
  \lambda_w=\lambda_{v\cdot x_k^{e_k}}
  =\lambda_{v \cdot \lambda_v\lambda_v^{-1}(x_k^{e_k})}
   =\lambda_{v \circ \lambda_v^{-1}(x_k^{e_k})}
   =\lambda_{v} \lambda_{\lambda_{\bar v}(x_k)^{e_k}}
   =\lambda_{v} \lambda_{\lambda_{\bar v}(x_k)}^{e_k}=
   \lambda_{x_1}^{e_1}\cdots\lambda_{x_{k-1}}^{e_{k-1}}\lambda_{x_k}^{e_k}
 \]
and $\lambda$ is a homomorphism with respect to~$\cdot$.

Now $(X,\hat\sigma,\hat\tau)$ is $2$-reductive as well, according to Proposition~\ref{prop:inv_2red}. Hence, for the skew left brace $G(X,\hat r)$, we have $\hat\lambda_{a\cdot_{op} b}=\hat\lambda_{b\cdot_{op} a}=\hat\lambda_a\hat\lambda_b$, for all~$a,b\in G(X,\hat r)$. According to Theorem~\ref{thm:disrbilred}, we have $\rho_a^{-1}=\hat\lambda_a$ and therefore we obtain $\rho_{a\cdot b}=\hat\lambda_{b\cdot_{op} a}^{-1}=\hat\lambda_b^{-1}\hat\lambda_a^{-1}=\rho_b\rho_a$ and analogously $\rho_{b\cdot a}=\rho_b\rho_a$.
\end{proof}

\end{document}